\documentclass[11pt,final]{article}
\usepackage[textwidth=5.5in,textheight=9.0in,centering]{geometry}
\usepackage{titlesec,titletoc}
\usepackage{etex}
\usepackage{mathtools}
\usepackage{pdfpages}
\usepackage{url}
\usepackage{amsmath}
\usepackage{amsxtra,amscd}
\usepackage{amsfonts}
\usepackage{amssymb}
\usepackage{graphicx}
\usepackage[amsmath,hyperref,thmmarks]{ntheorem}
\usepackage{natbib}
\usepackage{verbatim}
\usepackage{mathptmx}
\usepackage{enumitem}
\usepackage[notcite,notref]{showkeys}
\usepackage{url}
\theoremnumbering{arabic}
\theoremheaderfont{\scshape}
\RequirePackage{latexsym}
\theorembodyfont{\slshape}
\theoremseparator{.}
\newtheorem{X}{X}[section]

\newtheorem{conjecture}[X]{Conjecture}
\newtheorem{corollary}[X]{Corollary}

\newtheorem{lemma}[X]{Lemma}
\newtheorem*{*lemma}{Lemma}
\newtheorem{proposition}[X]{Proposition}
\newtheorem*{*proposition}{Proposition}

\newtheorem{theorem}[X]{Theorem}
\theorembodyfont{\upshape}
\newtheorem{aside}[X]{Aside}
\newtheorem{definition}[X]{Definition}
\newtheorem{rconjecture}[X]{Rationality Conjecture}
\newtheorem*{*definition}{Definition}

\newtheorem*{*example}{Example}

\newtheorem{remark}[X]{Remark}
\newtheorem{plain}[X]{}

\theorembodyfont{\footnotesize}

\theorembodyfont{\normalsize}
\theoremstyle{nonumberplain}
\theoremheaderfont{\sc}
\theorembodyfont{\normalfont}
\theoremsymbol{\ensuremath{_\Box}}
\RequirePackage{amssymb}
\newtheorem{proof}{Proof}
\qedsymbol{\ensuremath{_\Box}}
\theoremclass{LaTeX}
\makeindex
\let\oldvec=\vec
\renewcommand{\vec}[1]{\reflectbox{\ensuremath{\oldvec{\reflectbox{\ensuremath{#1}}}}}}

\contentsmargin{2.0em}
\dottedcontents{section}[2.3em]{}{1.8em}{1pc}
\dottedcontents{subsection}[6.0em]{}{1.8em}{1pc}
\titleformat*{\subsection}{\large\scshape}
\titleformat*{\subsubsection}{\slshape}

\usepackage{titlesec}                                                                            
\titleformat*{\section}{\LARGE\bfseries}
\titleformat*{\subsection}{\Large\itshape}
\titleformat*{\subsubsection}{\scshape}%\large
\titleformat*{\paragraph}{\itshape}

\usepackage{enumitem}
\setitemize[1]{label=$\diamond$\hspace{0.07in}}
\setlist{nolistsep}
\usepackage[nottoc]{tocbibind}

\usepackage{array}

\usepackage{tikz}
\usetikzlibrary{matrix,arrows,positioning,decorations.pathmorphing}
\usepackage{tikz-cd}

\usepackage{etex}

\usepackage{mathtools}

\usepackage[overload]{textcase}

\newcommand{\eb}[1]{{\itshape\bfseries#1}}
\renewcommand{\emph}{\eb}

\usepackage{natbib}
\let\cite\citealt
\newcommand{\bcomment}{\begin{comment}}\newcommand\ecomment{\end{comment}}
\newcommand{\bfootnotesize}{\begin{footnotesize}}\newcommand\efootnotesize{\end{footnotesize}}
\newcommand{\bquote}{\begin{quote}}\newcommand\equote{\end{quote}}
\newcommand{\bsmall}{\begin{small}}\newcommand\esmall{\end{small}}
\newcommand{\btable}{\begin{table}}\newcommand{\etable}{\end{table}}

\newcommand{\edocument}{\end{document}}

\newcommand{\tableoc}{\tableofcontents}%* removes "contents" from TOC

\def\1{{1\mkern-7mu1}}

\DeclareMathOperator{\Aut}{Aut}

\DeclareMathOperator{\Br}{Br}

\DeclareMathOperator{\Coker}{Coker}

\DeclareMathOperator{\End}{End}
\DeclareMathOperator{\ev}{ev}
\DeclareMathOperator{\Ext}{Ext}

\DeclareMathOperator{\Gal}{Gal}

\DeclareMathOperator{\Hom}{Hom}
\DeclareMathOperator{\id}{id}

\DeclareMathOperator{\Ker}{Ker}

\DeclareMathOperator{\ob}{ob}

\DeclareMathOperator{\rank}{rank}

\DeclareMathOperator{\Spec}{Spec}

\DeclareMathOperator{\Tr}{Tr}
\newcommand{\Mod}{\mathsf{Mod}}%modules
\newcommand{\Mot}{\mathsf{Mot}}%motives

\begin{document}

\title{Motivic complexes and special values of zeta functions}
\author{James S. Milne
\and Niranjan Ramachandran\thanks{Partly supported by NSF and Graduate Research
Board (UMD)}}
\date{November 13, 2013}
\maketitle

\begin{abstract}
Beginning with the conjecture of Artin and Tate in 1966, there has been a
series of successively more general conjectures expressing the special values
of the zeta function of an algebraic variety over a finite field in terms of
other invariants of the variety. In this article, we present the ultimate such
conjecture, and provide evidence for it. In particular, we enhance Voevodsky's
$\mathbb{Z}{}[1/p]$-category of \'{e}tale motivic complexes with a
$p$-integral structure, and show that, for this category, our conjecture
follows from the Tate and Beilinson conjectures. As the conjecture is stated
in terms of motivic complexes, it (potentially) applies also to algebraic
stacks, log varieties, simplicial varieties, etc..

\end{abstract}

\tableoc

\section{Introduction}

\subsection{Motivating examples}

We begin by reviewing two statements from the 1960s concerning the special
values of the zeta functions of varieties over finite fields. Our goal has
been to find the ultimate generalization of these statements, and to provide
persuasive evidence for it.

Recall that the zeta function of an algebraic variety $X$ over a finite field
$\mathbb{F}{}_{q}$ is the formal power series $Z(X,t)\in\mathbb{Q}{}[[t]]$
such that%
\begin{equation}
\log(Z(X,t))=\sum_{n>0}\frac{N_{n}t_{n}}{n} \label{eq1}%
\end{equation}
with $N_{n}=\#(X(\mathbb{F}{}_{q^{n}}))$, and that Dwork
(1960)\nocite{dwork1960} proved that $Z(X,t)\in\mathbb{Q}{}(t)$.

Let $X$ be a smooth projective surface over a finite field $\mathbb{F}{}_{q}$.
The N\'{e}ron-Severi group $\mathrm{NS}(X)$ of $X$ is finitely generated, and
the Tate conjecture says that its rank $\rho$ is the order of the pole of
$Z(X,t)$ at $t=q^{-1}$. Write%
\[
Z(X,t)=\frac{P_{1}(X,t)P_{3}(X,t)}{(1-t)P_{2}(X,t)(1-q^{2}t)},\quad
P_{i}(X,t)\in\mathbb{Z}{}[t].
\]
Then the Artin-Tate conjecture says that the Brauer group of $X$ is finite,
and that its order $[\Br(X)]$ satisfies%
\[
\lim_{t\rightarrow q^{-1}}\frac{P_{2}(X,t)}{(1-qt)^{\rho}}=\frac{[\Br(X)]\cdot
D}{q^{\alpha(X)}\cdot\lbrack\mathrm{NS}(X)_{\mathrm{tors}}]^{2}}%
\]
where $D$ is the discriminant of the intersection pairing on $\mathrm{NS}(X)$
and%
\[
\alpha(X)=\chi(X,\mathcal{O}{}_{X})-1+\dim(\text{\textrm{PicVar}}(X)).
\]

Now consider two abelian varieties $A$ and $B$ over a finite field, and write%
\begin{align*}
Z(A,t)  &  =\frac{P_{1}(A,t)\cdots}{(1-t)P_{2}(A,t)\cdots},\quad
P_{1}(A,t)=\prod\nolimits_{i}(1-a_{i}t)\\
Z(B,t)  &  =\frac{P_{1}(B,t)\cdots}{(1-t)P_{2}(B,t)\cdots},\quad
P_{1}(B,t)=\prod\nolimits_{j}(1-b_{j}t)\text{.}%
\end{align*}
Weil showed that $\Hom(A,B)$ is a finitely generated $\mathbb{Z}{}$-module,
and the Tate conjecture (proved by Tate in this case) says that%
\[
\rank(\Hom(A,B))=\#\{(i,j)\mid a_{i}=b_{j}\}.
\]
According to \cite{milne1968}, the group $\Ext^{1}(A,B)$ is finite, and its
order satisfies%
\[
\prod_{a_{i}\neq b_{j}}\left(  1-\frac{a_{i}}{b_{j}}\right)  =[\Ext^{1}%
(A,B)]\cdot D\cdot q^{-\dim(A)-\dim(B)}%
\]
where $D$ is the discriminant of the pairing%
\[
\Hom(A,B)\times\Hom(B,A)\overset{\circ}{\longrightarrow}%
\End(A)\overset{\mathrm{trace}}{\longrightarrow}\mathbb{Z}{}\text{.}%
\]

\subsection{Statement of the conjecture}

Let $k$ be a perfect field. Throughout the first six sections of the article,
$\mathsf{DM}(k)$ is a triangulated category\footnote{In fact, we work
throughout with pretriangulated differential graded categories. This aspect
will be ignored in the Introduction} of \textquotedblleft motivic
complexes\textquotedblright\ equipped with exact \textquotedblleft
realization\textquotedblright\ functors
\begin{align*}
r_{\ell}\colon\mathsf{DM}(k)  &  \rightarrow\mathsf{D}_{c}^{b}(k,\mathbb{Z}%
{}_{\ell})\text{, all }\ell\neq p,\\
r_{p}\colon\mathsf{DM}(k)  &  \rightarrow\mathsf{D}_{c}^{b}(R)
\end{align*}
where $\mathsf{D}_{c}^{b}(k,\mathbb{Z}{}_{\ell})$ is the $\ell$-adic
triangulated category (Ekedahl; see \S 4) and $\mathsf{D}_{c}^{b}(R)$ is the
triangulated category of coherent complexes of graded modules over the Raynaud
ring (Ekedahl-Illusie-Raynaud; see \S 5). There are Tate twists in
$\mathsf{DM}(k)$, compatible with the realization functors. We require that
$\mathsf{DM}(k)$ have an internal Hom, $R\underline{\Hom}(-,-)$. We do not
require $\mathsf{DM}(k)$ to have a $t$-structure. The $\Ext$ of two objects
$M,N$ in $\mathsf{DM}(k)$ is defined by the usual formula%
\[
\Ext^{j}(M,N)=\Hom_{\mathsf{DM}(k)}(M,N[j]).
\]

In \S 7 (resp. \S 8) we construct a candidate for $\mathsf{DM}(k)$ based on
Voevodsky's category of geometric motives (resp. the conjectural theory of
rational Tate classes).

Now let $k$ be a finite field with $q$ elements. Using the finiteness of $k$,
we construct (in \S 2) a canonical complex%
\[
E(M,N,r)\colon\quad\cdots\rightarrow\Ext^{j-1}(M,N(r))\rightarrow
\Ext^{j}(M,N(r))\rightarrow\Ext^{j+1}(M,N(r))\rightarrow\cdots
\]
of abelian groups for each pair $M,N$ in $\mathsf{DM}(k)$. Here $(r)$ denotes
the Tate twist in $\mathsf{DM}(k)$.

We expect that each object $P$ of $\mathsf{DM}(k)_{\mathbb{Q}{}}$ has a zeta
function $Z(P,t)\in\mathbb{Q}{}(t)$ compatible with the realization functors
(see \S 3).

Attached to each $P$ in $\mathsf{D}_{c}^{b}(R)$, there is a bounded complex
$R_{1}\otimes_{R}^{L}P$ of graded $k$-vector spaces whose cohomology groups
have finite dimension. The Hodge numbers $h^{i,j}(P)$ of $P$ are defined to be
the dimensions of the $k$-vector spaces $H^{j}(R_{1}\otimes_{R}^{L}P)^{i}$
(see \S 5).

\begin{conjecture}
\label{a1}Let $M,N\in\mathsf{DM}(k)$ and let $r\in\mathbb{Z}{}$. Let
$P=R\underline{\Hom}(M,N)$.

\begin{enumerate}
\item The groups $\Ext^{j}(M,N(r))$ are finitely generated $\mathbb{Z}{}%
$-modules for all $j$, and the alternating sum of their ranks is zero.

\item The zeta function $Z(P,t)$ of $P$ has a pole at $t=q^{-r}$ of order%
\[
\rho=\sum\nolimits_{j}(-1)^{j+1}\cdot j\cdot\rank_{\mathbb{Z}{}}%
(\Ext^{j}(M,N(r))).
\]

\item The cohomology groups of the complex $E(M,N,r)$ are finite, and the
alternating product $\chi^{\times}(M,N(r))$ of their orders satisfies%
\[
\left\vert \lim_{t\rightarrow q^{-r}}Z(P,t)(1-q^{r}t)^{\rho}\right\vert
=\chi^{\times}(M,N(r))\cdot q^{\chi(P,r)}%
\]
where%
\[
\chi(P,r)=\sum_{i,j\,\,(i\leq r)}(-1)^{i+j}(r-i)h^{i,j}(r_{p}(P))\text{.}%
\]
\noindent
\end{enumerate}
\end{conjecture}

Assume that $\mathsf{\mathsf{DM}}(k)$ has a tensor structure, and let $\1$ be
the identity object for this structure. Then $R\underline{\Hom}(\1,N)\simeq
N$. As we now explain, when we take $M$ to be $\1$, the statement of
Conjecture \ref{a1} simplifies, and it doesn't require the existence of
internal Homs.

We define the \emph{absolute} \emph{cohomology groups} of $P\in\ob\mathsf{DM}%
(k)$ by%
\[
H_{\mathrm{abs}}^{j}(P,r)=\Hom_{\mathsf{DM}(k)}(\1,P[j](r))
\]
(cf. \cite{deligne1994}, 3.2). With $(M,N)=(\1,P)$, the complex $E(M,N,r)$
becomes a complex%
\[
H_{\mathrm{abs}}^{\bullet}(P,r)\colon\quad\cdots\rightarrow H_{\mathrm{abs}%
}^{j-1}(P,r)\rightarrow H_{\mathrm{abs}}^{j}(P,r)\rightarrow H_{\mathrm{abs}%
}^{j+1}(P,r)\rightarrow\cdots.
\]
and Conjecture \ref{a1} becomes the following statement.

\begin{conjecture}
\label{a2}

\begin{enumerate}
\item The groups $H_{\mathrm{abs}}^{j}(P,r)$ are finitely generated
$\mathbb{Z}{}$-modules for all $j$, and the alternating sum of their ranks is zero.

\item The zeta function $Z(P,t)$ of $P$ has a pole at $t=q^{-r}$ of order%
\[
\rho=\sum\nolimits_{j}(-1)^{j+1}\cdot j\cdot\rank_{\mathbb{Z}{}}\left(
H_{\mathrm{abs}}^{j}(P,r)\right)  .
\]

\item The cohomology groups of the complex $H_{\mathrm{abs}}^{\bullet}(P,r)$
are finite, and the alternating product $\chi^{\times}(P,r)$ of their orders
satisfies%
\[
\left\vert \lim_{t\rightarrow q^{-r}}Z(P,t)(1-q^{r}t)^{\rho}\right\vert
=\chi^{\times}(P,r)\cdot q^{\chi(P,r)}%
\]
where
\[
\chi(P,r)=\sum_{i,j\,\,(\,i\leq r)}(-1)^{i+j}(r-i)\cdot h^{i,j}(r_{p}P).
\]

\end{enumerate}
\end{conjecture}

\subsection{Examples}

In these examples, we take $\mathsf{DM}(k)$ to be the category defined in \S 7.

\begin{plain}
\label{a3}Let $X$ be a smooth projective surface over $k=\mathbb{F}{}_{q}$.
The terms of the complex $E(\1,hX,1)$ are finitely generated $\mathbb{Z}%
$-modules, and all are finite except for $E^{2}$ and $E^{3}$, which both have
rank $\rho=\rank(\mathrm{NS}(X))$. Therefore (a) of (\ref{a1}) is true, and
(b) is the Tate conjecture for $X$. Modulo torsion, the map $E^{2}%
\overset{d}{\longrightarrow}E^{3}$ can be identified with the map
$\mathrm{NS}(X)\rightarrow\Hom(\mathrm{NS}(X),\mathbb{Z})$ defined by the
intersection pairing, whose cokernel has order $D$. In this case, (c) of
(\ref{a1}) essentially becomes the Artin-Tate conjecture.
\end{plain}

\begin{plain}
\label{a4}Let $M=h_{1}A$, $N=h_{1}B$, and $P=R\underline{\Hom}(h_{1}A,h_{1}B)$
where $A$ and $B$ are abelian varieties over $k=\mathbb{F}{}_{q}$. Then
$Z(P,t)=\prod_{i,j}(1-\frac{a_{i}}{b_{j}}t)$, and Conjecture \ref{a1}
essentially becomes the statement for $A$ and $B$ discussed
above.\footnote{This should be taken with a grain of salt: the $\Ext$'s are in
different categories.}
\end{plain}

\begin{plain}
\label{a4a}Let $X$ be a smooth projective variety over a finite field
$k=\mathbb{F}{}_{q}$, and let $\mathbb{Z}{}(r)$ be the complex of \'{e}tale
sheaves on $X$ given by Bloch's higher Chow groups (see the survey article
\cite{geisser2005}). Define the \emph{Weil-\'{e}tale motivic cohomology
groups} of $X$ to be%
\[
H_{\mathrm{mot}}^{j}(X,r)=H^{j}(X_{\text{W\'{e}t}},\mathbb{Z}{}(r))
\]
where W\'{e}t denotes the Weil-\'{e}tale topology (\cite{lichtenbaum2005}).
Geisser and Levine (2000, 2001)\nocite{geisserL2000}\nocite{geisserL2001} have
shown that $\mathbb{Z}{}(r)$ satisfies the \textquotedblleft Kummer
sequence\textquotedblright\ axiom (\cite{lichtenbaum1984}, (3), p.130;
\cite{milne1988}, (A2)$_{p}$, p.68). Assume that $X$ satisfies the Tate
conjecture and that, for some prime $l$, the ideal of $l$-homologically
trivial correspondences in $\mathrm{CH}^{\dim X}(X\times X)$ is nil. It then
follows from \cite{milne1986v} (and addendum) that Conjecture \ref{a2} holds
with the groups $H_{\mathrm{abs}}^{j}(X,r)$ replaced by $H_{\mathrm{mot}}%
^{j}(X,r)$.

The proof of this has four steps ($X$ is as above).

\begin{itemize}
\item From the Kummer sequence axiom, we get exact sequences%
\[
0\rightarrow H_{\mathrm{mot}}^{j}(X,r)_{l}\rightarrow H^{j}(X,\mathbb{Z}{}%
_{l}(r))\rightarrow T_{l}H_{\mathrm{mot}}^{j+1}(X,r)\rightarrow0.
\]
Here $M_{l}=\varprojlim_{n}M^{(l^{n})}$ is the $l$-adic completion of $M$, and
$T_{l}M=\varprojlim_{n}\Ker(M\overset{l^{n}}{\longrightarrow}M)$.

\item A theorem of Gabber shows that $H^{j}(X,\mathbb{Z}{}_{l}(r))$ is
torsion-free for almost all $l$ (i.e., for all but possibly finitely many).

\item Tate's conjecture for $X$ implies that the map%
\[
H_{\mathrm{mot}}^{j}(X,r)^{\prime}\otimes\mathbb{Z}{}_{l}\rightarrow
H^{j}(X,\mathbb{Z}{}_{l}(r))
\]
is an isomorphism. Here $M^{\prime}$ is the quotient of $M$ by its largest
uniquely divisible subgroup. Together with Gabber's theorem, and the local
results in \cite{milne1986v}, this implies Conjecture \ref{a2} for the groups
$H_{\mathrm{mot}}^{j}(X,r)^{\prime}$.

\item If the ideal of $l$-homologically trivial correspondences in
$\mathrm{CH}^{\dim X}(X\times X)$ is nil, then the groups $H_{\mathrm{mot}%
}^{j}(X,r)$ are finitely generated, and so Conjecture \ref{a2} holds for the
groups $H_{\mathrm{mot}}^{j}(X,r)$.
\end{itemize}
\end{plain}

\subsection{Notes}

\begin{plain}
\label{a6}Statement (b) of Conjecture \ref{a1} should be regarded as the Tate
conjecture for motivic complexes. In particular, Conjecture \ref{a1}
presupposes the usual Tate conjecture \textit{if} the category $\mathsf{DM}%
(k)$ is defined using algebraic classes, but there is no need to do this. We
envisage that Deligne's theory of absolute Hodge classes in characteristic
zero can be extended to a theory of rational Tate classes in characteristic
$p$ for which the Tate conjecture is automatically true (see \cite{milne2009}%
). Thus, we expect Conjecture \ref{a1} to be true (suitably interpreted) even
if the Tate conjecture proves false (see \S 8).
\end{plain}

\begin{plain}
\label{a7}There appears to be no direct relation between Conjecture \ref{a1}
and the Bloch-Kato conjecture. The latter applies to a motive over a global
field (see, for example, \cite{fontaine1992}), whereas our conjecture applies
to a complex of motives over a finite field.

Consider, for example, a smooth projective surface $X$ over a finite field
$k$. As noted above, our conjecture for the motive of $X$ essentially becomes
the conjecture of Artin and Tate for the surface $X$. Now let $f\colon
X\rightarrow C$ be a morphism from $X$ onto a curve $C$, and let $J$ be the
Jacobian of the generic fibre of $f$ (so $J$ is an abelian variety over the
function field $k(C)$ of $C$). Assume that $f$ has connected geometric fibres
and smooth generic fibre. The conjecture of Bloch and Kato for the motive
$h_{1}(J)$ over $k(C)$ is essentially the conjecture of Birch and
Swinnerton-Dyer for $J$ over $k(C)$. It is known that the Artin-Tate
conjecture for $X/k$ is equivalent to the Birch/Swinnerton-Dyer conjecture for
$J/k(C)$. Under some hypotheses on the map $f$, this was proved directly
(\cite{gordon1979}), but the general proof is difficult and indirect: it
proceeds by showing that each conjecture is equivalent to the Tate conjecture
for $X$ (\cite{milne1975}, \cite{katoT2003}). In other words, passing between
the two conjectures in this case is no easier than deducing the conjectures
separately from the Tate conjecture. We expect that a similar statement is
true for the Bloch-Kato conjecture and our conjecture in general.
\end{plain}

\begin{plain}
\label{a8}For some background and history to these questions, see
\cite{milne2013}.
\end{plain}

\begin{plain}
\label{a8a}Let $P$ be the motivic complex $\1[-j_{0}]$ over $\mathbb{F}{}_{q}%
$. Then%
\[
H_{\mathrm{abs}}^{j}(\bar{P},r)\overset{\textup{{\tiny def}}}{=}%
\Hom_{\mathsf{DM}(\mathbb{F}{})}(\1,\1(r)[j-j_{0}])
\]
is zero except that $H_{\mathrm{abs}}^{j_{0}}(\bar{P},r)=\mathbb{Z}{}$. From
the spectral sequence%
\[
H^{i}(\Gamma_{0},H_{\mathrm{abs}}^{j}(\bar{P},r))\implies H_{\mathrm{abs}%
}^{i+j}(P,r)
\]
we find that%
\[
H_{\mathrm{abs}}^{j_{0}}(P,r)\simeq\mathbb{Z}{}\simeq H^{j_{0}+1}(P,r)
\]
and the remaining groups are zero. As $Z(P(r),t)=(1-q^{r}t)^{(-1)^{j_{0}+1}}$,
we see that Conjecture \ref{a2} is true for $P$.

For a general $P\in\ob\mathsf{DM}(\mathbb{F}{}_{q})$, the Tate and other
conjectures predict that there should be a distinguished triangle%
\[
P_{0}\rightarrow P\rightarrow P_{1}\rightarrow P_{0}[1]
\]
with $P_{0}$ a direct sum of motivic complexes of the form $\1[-j_{0}]$ and
$P_{1}$ such that $Z(P_{1},t)$ has no pole or zero at $t=q^{-r}$ and
$H_{\mathrm{abs}}^{j}(P_{1},r)$ is finite for all $j$. From this, (a) and (b)
of Conjecture \ref{a2} follow.
\end{plain}

\begin{plain}
\label{a8b}Conjecture \ref{a2} is a special case of Conjecture \ref{a1}, and
so Conjecture \ref{a1} implies Conjecture \ref{a2}. On the other hand, for
$M,N$ in $\mathsf{DM}(k)$, we should have%
\[
R\Hom(M,N)\simeq R\Hom(\1,R\underline{\Hom}(M,N)),
\]
and so Conjecture \ref{a2} for $P=R\underline{\Hom}(M,N)$ implies Conjecture
\ref{a1} for $M,N$. Nevertheless, it is convenient to have both forms of the conjecture.
\end{plain}

\subsection{Evidence for the conjectures}

\subsubsection{The local versions are true}

The statement of Conjecture \ref{a1} uses few properties of the triangulated
category $\mathsf{\mathsf{DM}}(k)$. In fact, the same arguments lead to
similar conjectures for $\mathsf{D}_{c}^{b}(k,\mathbb{Z}{}_{\ell})$ and
$\mathsf{D}_{c}^{b}(R)$. Assume that the Frobenius maps are semisimple. Then
the conjecture for $\mathsf{D}_{c}^{b}(\mathbb{Z}{}_{\ell})$ is easy to prove
(see \S 4 below). That for $\mathsf{D}_{c}^{b}(R)$ is less easy, but is proved
in \cite{milneR2013} (see \S 5 below).

The realization functors $r_{\ell}$ and $r_{p}$ define maps%
\begin{align*}
\Hom_{\mathsf{DM}(k)}(M,N)  &  \rightarrow\Hom_{\mathsf{D}_{c}^{b}%
(k,\mathbb{Z}{}_{\ell})}(r_{\ell}M,r_{\ell}N)\\
\Hom_{\mathsf{DM}(k)}(M,N)  &  \rightarrow\Hom_{\mathsf{D}_{c}^{b}(R)}%
(r_{p}M,r_{p}N)
\end{align*}
($M,N\in\ob(\mathsf{\mathsf{DM}}(k))$), and hence maps%
\begin{align*}
r_{\ell}(M,N)\colon\Hom_{\mathsf{DM}(k)}(M,N)\otimes\mathbb{Z}{}_{\ell}  &
\rightarrow\Hom_{\mathsf{D}_{c}^{b}(k,\mathbb{Z}{}_{\ell})}(r_{\ell}M,r_{\ell
}N)\\
r_{p}(M,N)\colon\Hom_{\mathsf{DM}(k)}(M,N)\otimes\mathbb{Z}{}_{p}  &
\rightarrow\Hom_{\mathsf{D}_{c}^{b}(R)}(r_{p}M,r_{p}N).
\end{align*}
When $k$ is finite, we expect the maps $r_{l}(M,N)$ to be isomorphisms (see
below). If so, then our local results imply a statement that is only a little
weaker than Conjecture \ref{a1} (see \ref{d5}).

\subsubsection{Voevodsky motivic complexes}

From Example \ref{a4a}, it is clear that, in order for Conjecture \ref{a2} to
be true, we must have%
\[
\Hom_{\mathsf{DM}(k)}(\1,M(X)(r)[j])\simeq H_{\mathrm{mot}}^{j}(X,r)
\]
for all smooth projective varieties $X$, where $M(X)$ is the motivic complex
of $X$ and $H_{\mathrm{mot}}^{j}(X,r)$ is the \textit{Weil-\'{e}tale} motivic
cohomology group. This excludes Voevodsky's category
$\mathsf{\mathsf{\mathsf{DM}_{\mathrm{gm}}}}(k)^{\mathrm{opp}}$, which gives
the \textit{Zariski} motivic cohomology groups.

Instead, we need to look at $\left(  \mathsf{DM_{\mathrm{et}}}(k)\right)
^{\mathrm{opp}}$. This is known to give the \textit{\'{e}tale} motivic
cohomology groups \textit{except for the }$p$\textit{-part}; in fact,
$\mathsf{DM_{\mathrm{et}}}(k)$ is a $\mathbb{Z}{}[p^{-1}]$-linear category. In
order to obtain a category that is also correct for $p$-torsion, we define an
exact functor $\left(  \mathsf{DM}_{\mathrm{et}}(k)\right)  _{\mathbb{Q}{}%
}^{\mathrm{opp}}\rightarrow\mathsf{D}_{c}^{b}(R)_{\mathbb{Q}{}}$ and form the
\textquotedblleft fibred product\textquotedblright\ category%
\[
\begin{tikzcd}
\mathsf{DM}(k)\arrow{d}\arrow{r}&{\left(\mathsf{DM}_{\mathrm{et}}(k)\right)^\mathrm{opp}}\arrow{d}\\
\mathsf{D}_{c}^{b}(R)\arrow{r}&\mathsf{D}_{c}^{b}(R)_{\mathbb{Q}}.%
\end{tikzcd}
\]
This category gives the full \textit{\'{e}tale} motivic cohomology groups,
which is correct over $\mathbb{F}{}$, but over $\mathbb{F}{}_{q}$ we want the
\textit{Weil-\'{e}tale} motivic cohomology groups. There are two more-or-less
equivalent ways of achieving this: repeat the above construction with
Lichtenbaum's Weil-\'{e}tale topology, or define $\mathsf{DM}_{\mathrm{W\acute
{e}t}}(\mathbb{F}{}_{q})$ in terms of $\mathsf{DM_{\mathrm{et}}}(\mathbb{F}%
{})$. See \S 7.

Once $\mathsf{DM}(k)$ has been defined, we can proceed as in Example \ref{a4a}
(we use the same notations).

\begin{itemize}
\item From the rigidity theorem of Suslin and Voevodsky (\cite{voevodsky2000},
3.3.3) and its $p$ counterpart, we obtain exact sequences%
\[
0\rightarrow\Ext^{j}(M,N(r))_{l}\rightarrow\Ext^{j}(r_{l}M,r_{l}%
N(r))\rightarrow T_{l}\Ext^{j}(M,N(r))\rightarrow0
\]
for all $l$.

\item Assuming that the category $\mathsf{DM}(k)_{\mathbb{Q}{}}$ is
semisimple, we show that, for almost all $l$, the group $\Ext^{j}(r_{l}%
M,r_{l}N(r))$ is torsion-free.

\item Tate's conjecture for smooth projective varieties implies that the map%
\[
\Ext^{j}(M,N)^{\prime}\otimes\mathbb{Z}{}_{l}\rightarrow\Ext^{j}(r_{l}%
M,r_{l}N)
\]
is an isomorphism. Together with the preceding statement and our local results
(\S 4,\S 5), this implies Conjecture \ref{a2} for the groups $\Ext^{j}%
(M,N)^{\prime}$.

\item The groups $\Ext^{j}(M,N)$ contain no nontrivial uniquely divisible
subgroups if and only if some realization functor $r_{l}$ is faithful (in
which case, they all are). This is true, for example, if the Beilinson
conjecture\footnote{The Beilinson conjecture says that, for smooth projective
varieties over a finite field, rational equivalence coincides with numerical
equivalence with $\mathbb{Q}{}$-coefficients.} holds for smooth projective varieties.
\end{itemize}

Of course, a theorem that assumes the Tate conjecture for \textit{all} smooth
projective varieties over $\mathbb{F}{}_{q}$ is not of much value, since we
may never be able to prove this. Fortunately, our results are more precise
(see \S \S 6,7,8).

\subsection{Motivic complexes for rational Tate classes.}

If the category $\mathsf{DM}(k)$ is defined using algebraic cycles, then
Conjecture \ref{a1} requires the Tate conjecture. However, we expect that the
theory of rational Tate classes will provide a (unique) extension of Deligne's
theory of absolute Hodge cycles to mixed characteristic. Assuming the
rationality conjecture (8.1), which is weaker than the Hodge conjecture for CM
abelian varieties, we construct in \S 8 a category $\mathsf{DM}(\mathbb{F}%
{}_{q})$ of motivic complexes for which the conjectures are automatically true.

\subsection{Notations}

\textit{We use }$\ell$\textit{ for a prime number }$\neq p$\textit{, and }%
$l$\textit{ for a prime number, possibly }$p$. The $l$-adic absolute value is
normalized so that%
\[
\left\vert a\right\vert _{l}^{-1}=(\mathbb{Z}{}_{l}\colon a\mathbb{Z}{}%
_{l}),\quad a\in\mathbb{Z}{}_{l}.
\]
For an object $M$ of an abelian category, $M^{(n)}$ denotes the cokernel of
multiplication by $n$ on $M$. For an object $M$ of a triangulated category,
$M^{(n)}$ denotes the cone of $M\overset{n}{\longrightarrow}M$.

\section{The complex $E(M,N,r)$}

\subsection{Differential graded enhancements}

Recall that a \emph{graded category} is an additive category equipped with
gradations%
\[
\Hom(A,B)=\bigoplus\nolimits_{n\in\mathbb{Z}{}}\Hom^{n}(A,B)
\]
on the Hom groups that are compatible with composition of morphisms; in
particular, $\id_{A}\in\Hom^{0}(A,A)$ . Such a category is a
\emph{differential graded (dg) category} if, in addition, it is equipped with
differentials $d\colon\Hom^{n}(A,B)\rightarrow\Hom^{n+1}(A,B)$ such that
$d\circ d=0$ and%
\[
d(g\circ f)=dg\circ f+(-1)^{\deg g}(g\circ df)
\]
whenever $g$ is homogeneous and $g\circ f$ is defined. The \emph{homotopy
category }$\mathrm{Ho}(\mathcal{C)}$ of a dg category $\mathcal{C}{}$ has the
same objects a $\mathcal{C}{}$ but%
\[
\Hom_{\mathrm{Ho}(\mathcal{C}{})}(A,B)=H^{0}(\Hom_{\mathcal{C}{}}^{\bullet
}(A,B))\text{.}%
\]

Let $\mathsf{C}$ be a triangulated category. By a \emph{dg enhancement} of
$\mathsf{C},$ we mean a dg category $\mathcal{C}{}$ such that $\mathrm{Ho}%
(\mathcal{C}{})=\mathsf{C}$ and%
\begin{equation}
\Hom_{\mathcal{C}}^{\bullet}(A,B[m])\simeq\Hom_{\mathcal{\mathcal{C}{}}%
}^{\bullet}(A,B)[m] \label{eq19}%
\end{equation}
for $A,B\in\ob\mathsf{C}$, and%
\[
\Hom_{\mathcal{C}{}}(C,\text{Cone}(f))\simeq\text{Cone}(\Hom_{\mathcal{C}%
}(C,A)\overset{f\circ-}{\longrightarrow}\Hom(C,B))
\]
for all $C\in\ob\mathcal{C}{}$ and morphism $f\colon A\rightarrow B$ in
$\mathcal{C}{}$ such that $d\circ f=0$. More generally, we allow a dg
enhancement of $\mathsf{C}$ to be a pretriangulated dg category $\mathcal{C}%
{}$ together with an equivalence to $\mathsf{C}$ from the triangulated
category associated with $\mathcal{C}{}$.

When $\mathsf{C}$ has a fixed dg enhancement of $\mathcal{C}{}$, we write%
\begin{equation}
R\Hom_{\mathsf{C}}^{\bullet}(A,B)=\Hom_{\mathsf{\mathcal{C}{}}}^{\bullet
}(A,B). \label{eq18}%
\end{equation}
Thus, $R\Hom_{\mathsf{C}}^{\bullet}(A,B)$ is a complex such that%
\[
H^{n}(R\Hom_{\mathsf{C}}^{\bullet}(A,B))=\Hom_{\mathsf{C}}%
(A,B[n])\overset{\textup{{\tiny def}}}{=}\Ext^{n}(A,B).
\]

\subsection{The cohomology of $\Gamma_{0}$}

Let $\Gamma_{0}$ be the free abelian group generated by a single element
$\gamma$ (thus $\Gamma_{0}\simeq\mathbb{Z}{}$), and let $\mathbb{Z}{}%
\Gamma_{0}$ be its group ring. For a $\Lambda$-module $M$, we let $M_{\ast}$
denote the corresponding co-induced module. Recall that this consists of the
maps from $\Gamma_{0}$ to $M$, and that $\tau\in\Gamma_{0}$ acts on $f\in
M_{\ast}$ according to the rule $(\tau f)(\sigma)=f(\sigma\tau)$. For each
$\Gamma_{0}$-module $M$, there is an exact sequence%
\begin{equation}
0\rightarrow M\longrightarrow M_{\ast}\overset{\alpha_{\gamma}%
}{\longrightarrow}M_{\ast}\rightarrow0, \label{eq2}%
\end{equation}
in which the first map sends $m\in M$ to the map $\sigma\mapsto\sigma m$ and
the second map sends $f\in M_{\ast}$ to $\sigma\mapsto f(\sigma\gamma)-\gamma
f(\sigma)$. Let $F$ denote the functor $M\mapsto M^{\Gamma_{0}}\colon
\Mod(\mathbb{Z}{}\Gamma_{0})\rightarrow\Mod(\mathbb{Z}{})$. The class of
co-induced $\Lambda\Gamma_{0}$-modules is $F$-injective, and so (\ref{eq2})
defines isomorphisms
\[
RF(M)\simeq F(M_{\ast}\overset{\alpha_{\gamma}}{\longrightarrow}M_{\ast
})\simeq(M\overset{1-\gamma}{\longrightarrow}M)
\]
in $\mathsf{D}^{+}(\mathbb{Z}{})$. For the second isomorphism, note that
$(M_{\ast})^{\Gamma_{0}}$ consists of the constant functions $\Gamma
_{0}\rightarrow M$, and that if $f$ is the constant function with value $m$,
then%
\[
(a_{\gamma}f)(\sigma)=f(\sigma\gamma)-\gamma f(\sigma)=m-\gamma m\text{, all
}\sigma\in\Gamma_{0}\text{.}%
\]

For every complex $X$ of $\mathbb{Z}{}\Gamma_{0}$-modules, there is an exact
sequence%
\begin{equation}
0\rightarrow X\rightarrow X_{\ast}\overset{\alpha_{\gamma}}{\longrightarrow
}X_{\ast}\rightarrow0 \label{eq3}%
\end{equation}
of complexes with $X_{\ast}^{j}=(X^{j})_{\ast}$ for all $j$. We deduce from
(\ref{eq3}) isomorphisms%
\[
RF(X)\simeq s(F(X_{\ast}\longrightarrow X_{\ast}))\simeq s(X\overset{1-\gamma
}{\longrightarrow}X)
\]
in $\mathsf{D}^{+}(\mathbb{Z}{})$ where $X\overset{1-\gamma}{\longrightarrow
}X$ is a double complex with $X$ as both its zeroth and first column and $s$
denotes the associated total complex. In other words,%
\begin{equation}
RF(X)[1]\simeq\text{Cone}(1-\gamma\colon X\rightarrow X)\text{.} \label{eq4}%
\end{equation}
From (\ref{eq4}), we get a long exact sequence%
\begin{equation}
\cdots\rightarrow H^{j-1}(X)\overset{1-\gamma}{\longrightarrow}H^{j-1}%
(X)\rightarrow R^{j}F(X)\rightarrow H^{j}(X)\overset{1-\gamma}{\longrightarrow
}H^{j}(X)\rightarrow\cdots. \label{eq5}%
\end{equation}

Note that%
\[
R^{1}F(\mathbb{Z}{})\simeq H^{1}(\Gamma_{0},\mathbb{\mathbb{Z}{}}%
)\simeq\Hom(\Gamma_{0},\mathbb{Z}{})\text{,}%
\]
which has a canonical element $\theta\colon\gamma\mapsto1$. We can regard
$\theta$ as an element of%
\[
\Ext^{1}(\mathbb{Z}{},\mathbb{Z}{}[1])\overset{\textup{{\tiny def}}%
}{=}\Hom_{\mathsf{D}^{+}(\mathbb{Z}{}\Gamma_{0})}(\mathbb{Z}{},\mathbb{Z}%
{}[1]).
\]
Thus, for $X$ in $\mathsf{D}^{+}(\mathbb{Z}{}\Gamma_{0})$, we obtain maps%
\begin{align}
\theta\colon X  &  \rightarrow X[1]\nonumber\\
R\theta\colon RF(X)  &  \rightarrow RF(X)[1]. \label{eq8}%
\end{align}
The second map is described explicitly by the following map of double
complexes:%
\[
\begin{tikzpicture}[baseline=(current bounding box.center), text height=1.5ex, text depth=0.25ex]
\node (a) at (0,0) {$RF(X)$};
\node (b) at (2,0) {};
\node (c) at (4,0) {$X$};
\node (d) at (6,0) {$X$};
\node (e) [below=of a] {$RF(X)[1]$};
\node (f)[below=of b] {$X$};
\node (g)[below=of c] {$X$};
\node (h)[below=of d] {};
\node at (2,-2.05) {$\scriptstyle{-1}$};
\node at (4,-2.05) {$\scriptstyle{0}$};
\node at (6,-2.05) {$\scriptstyle{1}$};
\draw[->,font=\scriptsize,>=angle 90]
(a) edge node[right]{$R\theta$} (e)
(c) edge node[right]{$\gamma$} (g)
(c) edge node[above]{$1-\gamma$} (d)
(f) edge node[above]{$1-\gamma$} (g);
\end{tikzpicture}
\]
For all $j$, the following diagram commutes%
\begin{equation}
\begin{tikzcd} R^{j}F(X)\arrow{d}\arrow{r}{d^{j}}& R^{j+1}F(X)\\ H^{j}(X)\arrow{r}{\id} & H^{j}(X)\arrow{u} \end{tikzcd} \label{eq6}%
\end{equation}
where $d^{j}=H^{j}(R\theta)$ and the vertical maps are those in (\ref{eq5}).
The sequence%
\begin{equation}
\cdots\rightarrow R^{j-1}F(X)\overset{d^{j-1}}{\longrightarrow}R^{j}%
F(X)\overset{d^{j}}{\longrightarrow}R^{j+1}F(X)\rightarrow\cdots\label{eq7}%
\end{equation}
is a complex because $R\theta\circ R\theta=0$.

\subsection{Construction of the complex $E(M,N,r)$}

Let $\bar{k}$ be an algebraic closure of $k$, and let $\Gamma_{0}%
\subset\Gal(\bar{k}/k)$ be the Weil group of $k$. We let $\gamma$ denote the
generator $x\mapsto x^{q}$ of $\Gamma_{0}$.

We require a dg enhancement $\mathcal{DM}(k)$ of $\mathsf{\mathsf{DM}}(k)$. In
particular, this means that there is a functor%
\[
R\Hom\colon\mathsf{DM}(k)^{\mathrm{opp}}\times\mathsf{DM}(k)\rightarrow
\mathsf{D}^{+}(\mathbb{Z}{})
\]
such that%
\[
\Hom_{\mathsf{\mathsf{DM}}(k)}(M,N[j])\simeq H^{j}(R\Hom(M,N))
\]
for $M,N\in\ob\mathsf{\mathsf{DM}}(k)$ and $j\in\mathbb{Z}{}$. We require that
$R\Hom$ be related to the internal Hom by%
\begin{equation}
R\Hom(\1,R\underline{\Hom}(M,N))\simeq R\Hom(M,N),\quad M,N\in
\ob\mathsf{\mathsf{DM}}(k)\text{.} \label{eq17}%
\end{equation}

We require that the functor%
\[
R\Hom(\1,-)\colon\mathsf{\mathsf{DM}}(k)\rightarrow\mathsf{D}^{+}(\mathbb{Z}%
{})
\]
factors through $\mathsf{D}^{+}(\mathbb{Z}{}\Gamma_{0})$; moreover, that this
factorization arises from a factorization of $M\rightsquigarrow
\Hom_{\mathsf{\mathsf{DM}}(k)}(\1,M)$ through $M\rightsquigarrow
\Hom_{\mathsf{\mathsf{DM}}(\bar{k})}(\1,\bar{M})$.\footnote{For the motivic
complex of a smooth projective variety $X$, this amounts to requiring the
existence of a spectral sequence%
\[
H^{i}(\Gamma_{0},H_{\mathrm{abs}}^{j}(\bar{X},r))\implies H_{\mathrm{abs}%
}^{i+j}(X,r)
\]
for each integer $r$.}

From (\ref{eq17}), we see that
\[
R\Hom\colon\mathsf{DM}(k)^{\mathrm{opp}}\times\mathsf{DM}(k)\rightarrow
\mathsf{D}^{+}(\mathbb{Z}{})
\]
factors through $RF\colon\mathsf{D}^{+}(\mathbb{Z}{}\Gamma_{0})\rightarrow
\mathsf{D}^{+}(\mathbb{Z)}$:
\[
\begin{tikzcd}
\mathsf{DM}(k)\arrow{r}\arrow[bend left=20]{rr}{R\Hom(\1,-)}
&\mathsf{D}^{+}(\mathbb{Z}\Gamma_{0})\arrow{r}{RF}
&\mathsf{D}^{+}(\mathbb{Z})\\
\mathsf{DM}(k)^{\mathrm{opp}}\times\mathsf{DM}(k)\arrow{u}{R\underline{\Hom}(-,-)}
\arrow{rru}[swap]{R\Hom(-,-)}
\end{tikzcd}
\]
Therefore, for each pair $M,N$ of objects of $\mathsf{\mathsf{DM}}(k)$, there
is a well-defined object $X$ of $\mathsf{D}(\mathbb{Z}{}\Gamma_{0})$ such that
$RFX=R\Hom(M,N)$, and%
\[
H^{j}(RFX)=H^{j}(R\Hom(M,N))\simeq\Hom_{\mathsf{\mathsf{DM}}(k)}%
(M,N[j])\overset{\textup{{\tiny def}}}{=}\Ext^{j}(M,N).
\]
Now (\ref{eq7}) becomes a complex%

\[
E(M,N,r)\text{:}\quad\cdots\longrightarrow\Ext^{j-1}(M,N(r))\overset{d^{j-1}%
}{\longrightarrow}\Ext^{j}(M,N(r))\overset{d^{j}}{\longrightarrow}%
\Ext^{j+1}(M,N(r))\longrightarrow\cdots\text{.}%
\]

\section{The zeta function of a motivic complex}

Throughout this section, $k$ is a perfect field.

Let $P$ be an object of $\mathsf{D}_{c}^{b}(k,\mathbb{Z}{}_{\ell})$. We can
think of $P$ as a bounded complex of $\mathbb{Z}{}_{\ell}$-modules with
finitely generated cohomology, equipped with a continuous action of $\Gamma$.
For an endomorphism $\alpha$ of $P$, we define $c_{\alpha}(t)$ to be the
alternating product of the characteristic polynomials of $\alpha$ acting on
the $\mathbb{Q}{}_{\ell}$-vector spaces $H^{i}(P)_{\mathbb{Q}{}_{\ell}}$:%
\[
c_{\alpha}(t)=\prod\nolimits_{i}\det(1-\alpha t|H^{i}(P)_{\mathbb{Q}{}_{\ell}%
})^{(-1)^{i+1}}\in\mathbb{Q}{}_{\ell}(t).
\]

Let $P$ be an object of $\mathsf{D}_{c}^{b}(R)$ and let $\alpha$ be an
endomorphism of $P$. We define $c_{\alpha}(t)$ to be the alternating product
of the characteristic polynomials of $\alpha$ acting on the isocrystals
$H^{i}(sP)_{K}$:%
\[
c_{\alpha}(t)=\prod\nolimits_{i}\det(1-\alpha t|H^{i}(sP)_{K})^{(-1)^{i+1}%
}\noindent\in\mathbb{Q}{}_{p}(t)
\]
(cf. \cite{demazure1972}, p.89).

\subsection{Determinants}

Let $\mathsf{T}$ be a triangulated category, and let $\mathsf{T}%
_{\mathrm{\approx}}$ denote the subcategory with the same objects but with
only the isomorphisms as morphisms. A \emph{determinant functor} on
$\mathsf{T}$ is a functor $f\colon\mathsf{T}_{\approx}\rightarrow\mathsf{P}$
from $\mathsf{T}$ to a Picard category satisfying certain natural conditions
(\cite{breuning2011}, 3.1). Every (essentially) small triangulated category
admits a universal determinant functor $f\colon\mathsf{T}_{\approx}%
\rightarrow\mathsf{P}$, which is unique up to a non-unique isomorphism (ibid.
\S 4). The automorphism group of an object $X$ of $\mathsf{P}$ is independent
of $X$ up to a well-defined isomorphism --- we denote it by $\pi
_{1}(\mathsf{P})$. The \emph{determinant} $\det(\alpha)$ of an automorphism
$\alpha$ of an object $P$ of $\mathsf{T}$ is defined to be the element
$f(\alpha)$ of $\Aut_{\mathsf{P}}(f(P))\simeq\pi_{1}(\mathsf{P}).$

Let $f\colon\mathsf{\mathsf{DM}}(k)_{\mathbb{Q}{}}\rightarrow\mathsf{P}$ be
the universal determinant functor for the triangulated category
$\mathsf{\mathsf{DM}}(k)_{\mathbb{Q}{}}$. Because $\mathsf{\mathsf{DM}%
}(k)_{\mathbb{Q}{}}$ is $\mathbb{Q}{}$-linear, $\pi_{1}(\mathsf{P}%
)\supset\mathbb{Q}{}^{\times}$.

\begin{proposition}
\label{a9}Let $P\in\ob\mathsf{\mathsf{DM}}(k)_{\mathbb{Q}{}}$, and let
$\alpha$ be an automorphism of $P_{\mathbb{Q}{}}$. If $\det(1-\alpha
n)\in\mathbb{Q}{}^{\times}$ for all $n\in\mathbb{Z}{}$, then there exists a
unique $c_{\alpha}(t)\in\mathbb{Q}{}(t)$ such that $c_{\alpha}(n)=\det
(1-\alpha n)$ for all $n\in\mathbb{Z}{}$. Moreover, $c_{\alpha}(t)=c_{r_{l}%
(\alpha)}(t)$ for all $l$.
\end{proposition}

\begin{proof}
Let $P,P_{1},Q,Q_{1}\in\mathbb{Q}{}[t]$ be such that $P(n)/Q(n)=P_{1}%
(n)/Q_{1}(n)$ for all $n\in\mathbb{Z}{}$. Then $P(t)Q_{1}(t)-P_{1}(t)Q(t)$ has
infinitely many roots, and so is zero. This proves the uniqueness of
$c_{\alpha}(t)$.

Fix an $l$ (possibly $p$). There is a commutative diagram%
\[
\begin{tikzcd}
\mathsf{\mathsf{DM}}(k)\arrow{r}\arrow{d}{\det}&\mathsf{D}_{c}^{b}(\mathbb{Z}_{l})\arrow{d}{\det}\\
\mathsf{P}\arrow[dashed]{r}&\mathsf{P}_{l}%
\end{tikzcd}
\]
in which the vertical arrows are universal. Let $c_{r_{l}(\alpha
)}(t)=P(t)/Q(t)$ with $P(t)$, $Q(t)\in\mathbb{Q}{}_{l}[t]$. Then
$P(n)/Q(n)=\det(1-\alpha n)$ for all $n\in\mathbb{Z}{}$. Let $P(t)=\sum
c_{i}t^{i}$ and $Q(t)=\sum d_{j}t^{j}$. Choose distinct rational numbers
$n_{1},\ldots,n_{r}$ with $r$ at least $\max(\deg(P),\deg(Q))$. Then
$(c_{1},c_{2},\ldots,d_{1},\ldots)$ is the unique solution of the system of
linear equations%
\[
\sum c_{i}n_{s}^{i}=\det(1-\alpha n_{s})\cdot\sum d_{j}n_{s}^{j},\quad
s=1,\ldots,r.
\]
As the coefficients of these linear equations lie in $\mathbb{Q}{}$, their
solution does also.
\end{proof}

\begin{aside}
\label{a10}The definition of the characteristic polynomial of an endomorphism
of a motivic complex in (\ref{a9}) follows that in \cite{milne1994}, 2.1, for
an endomorphism of a motive, which in turn follows that in \cite{weil1948},
IX, for an endomorphism of an abelian variety.
\end{aside}

\begin{aside}
\label{a11}What (conjecturally) is the $\mathsf{P}$ attached to
$\mathsf{\mathsf{DM}}(k)$?
\end{aside}

\subsection{Traces}

Let $P$ be an object in a tensor category. When $P$ has a dual $(P^{\vee
},\ev,\delta)$, we can define the \emph{trace} of an endomorphism $\alpha$ of
$P$ to be the composite of%
\[
\1\overset{\delta}{\longrightarrow}P\otimes P^{\vee}%
\xrightarrow{\mathrm{transpose}}P^{\vee}\otimes P\overset{\id\otimes
\alpha}{\longrightarrow}P^{\vee}\otimes P\overset{\ev}{\longrightarrow}\1.
\]
It is an element of $\End(\1)$.

Now assume that $\mathsf{\mathsf{DM}}(k)_{\mathbb{Q}{}}$ has the structure of
a rigid tensor category and that each of the realization functors $r_{l}$ is a
tensor functor. Define $c_{\alpha}(t)$ to be the power series satisfying
(\ref{eq1}) with $N_{n}=\Tr(\alpha^{n}|P)$. Then $c_{\alpha}(t)$ maps to
$c_{r_{l}(\alpha)}(t)$ under $\mathbb{Q}{}[[t]]\mapsto\mathbb{Q}{}_{l}[[t]]$
for all $l$ (including $l=p$). As\footnote{The condition that a power series
be a rational function is linear; see Bourbaki, Alg\`{e}bre, Chap. IV, \S 4,
Ex. 1.} $\mathbb{Q}{}(t)=\mathbb{Q}{}[[t]]\cap\mathbb{Q}{}_{l}(t)$, this shows
that $c_{\alpha}(t)\in\mathbb{Q}{}(t)$.

\subsection{Zeta functions}

Now assume that $k$ is finite, with $q=p^{a}$ elements, and let $\gamma$ be
the generator $x\mapsto x^{q}$ of $\Gal(\mathbb{F}{}/\mathbb{F}{}_{q})$. We
define the zeta function $Z(P,t)$ of an object $P$ of $\mathsf{D}_{c}%
^{b}(k,\mathbb{Z}{}_{\ell})$ (resp. $\mathsf{D}_{c}^{b}(R)$) to be $c_{\gamma
}(t)$ (resp. $c_{F^{a}}(t)$). Let $P$ be an object of $\mathsf{\mathsf{DM}%
}(k)$. We say that an element $Z(P,t)$ of $\mathbb{Q}{}(t)$ is the \emph{zeta
function} of $P$ if $Z(P,t)=Z(r_{l}(P),t)$ for all prime numbers $l$
(including $l=p$) --- clearly $Z(P,t)$ is unique if it exists.

Now assume that there is a Frobenius endomorphism $\pi$ of $\id_{\mathsf{DM}%
(k)}$ such that $r_{\ell}(\pi_{P})$ acts as $\gamma$ on $r_{\ell}(P)$ and
$r_{p}(\pi_{P})$ acts as $F^{a}$ on $r_{p}(P)$.

\begin{proposition}
\label{a9b}Under each of the following two hypotheses, every object of
$\mathsf{\mathsf{DM}}(k)$ has a zeta function:
\end{proposition}

\begin{enumerate}
\item for all $P\in\ob\mathsf{DM}(k)$ and $n\in\mathbb{Z}{}$, $\det(1-\pi
_{P}n)\in\mathbb{Q}{}^{\times}$;

\item the category $\mathsf{\mathsf{DM}}(k)_{\mathbb{Q}{}}$ has a rigid tensor
structure, and each of the realization functors $r_{l}$ is a tensor functor.
\end{enumerate}

\begin{proof}
Immediate consequence of the above discussion.
\end{proof}

\section{The local conjecture at $\ell$}

\begin{plain}
\label{a21}For a scheme $X$ of finite type over a field $k$ and a prime
$\ell\neq\mathrm{char}(k)$, we let $\mathsf{D}_{c}^{b}(X,\mathbb{Z}{}_{\ell})$
denote the triangulated category of bounded constructible $\mathbb{Z}{}_{\ell
}$-complexes on $X$ (\cite{ekedahl1990}). This can be constructed as follows.
Let $\Lambda=\mathbb{Z}{}_{\ell}$ and let $\Lambda_{n}=\mathbb{Z}{}_{\ell
}/\ell^{n}\mathbb{Z}{}_{\ell}$. The inverse systems%
\[
F=(\cdots\leftarrow F_{n}\leftarrow F_{n+1}\leftarrow\cdots)
\]
in which $F_{n}$ is a sheaf of $\Lambda_{n}$-modules on $X$ form an abelian
category $\mathcal{S}{}(X,\Lambda_{\bullet})$, whose derived category we
denote by $\mathsf{D}(X,\Lambda_{\bullet})$. An inverse system $F=(F_{n}%
)\in\ob\mathcal{S}{}(X,\Lambda_{\bullet})$ is said to be \emph{essentially
zero} if, for each $n$, there exists an $m\geq n$ such that the transition map
$F_{m}\rightarrow F_{n}$ is zero. A complex $K\in\ob\mathsf{D}(X,\Lambda
_{\bullet})$ is \emph{essentially zero} if each inverse system $H^{i}(K)$ is
essentially zero. Now $\mathsf{D}_{c}^{b}(X,\mathbb{Z}{}_{\ell})$ can defined
to be the full subcategory of $\mathsf{D}(X,\Lambda_{\bullet})$ consisting of
the complexes $M=(M_{n})_{n\in\mathbb{N}{}}$ such that the inverse system%
\[
\Lambda_{1}\otimes_{\Lambda_{\bullet}}^{L}M\overset{\textup{{\tiny def}}%
}{=}(\Lambda_{1}\otimes_{\Lambda_{n}}^{L}M_{n})_{n\in\mathbb{N}{}}%
\]
of complexes is isomorphic, modulo essentially zero complexes, to the constant
inverse system defined by an object of $\mathsf{D}_{c}^{b}(X,\Lambda_{1})$
(triangulated category of bounded complexes of $\Lambda_{1}$-sheaves on $X$
with constructible cohomology).
\end{plain}

\begin{plain}
\label{a21a}Let $\mathcal{S}{}(X,\mathbb{Z}{}_{\ell})$ denote the category of
sheaves of $\mathbb{Z}{}_{\ell}$-modules on $X_{\mathrm{et}}$. The obvious
functors%
\[
\mathcal{S}{}(X,\Lambda_{\bullet})\overset{\pi_{\ast}}{\longrightarrow
}\mathcal{S}{}(X,\mathbb{Z}{}_{\ell})\overset{\pi^{\ast}}{\longrightarrow
}\mathcal{S}{}(X,\Lambda_{\bullet})
\]
induce functors on the corresponding derived categories%
\[
\mathsf{D}^{+}(X,\Lambda_{\bullet})\overset{R\pi_{\ast}}{\longrightarrow
}\mathsf{D}^{+}(X,\mathbb{Z}{}_{\ell})_{\text{\textrm{litt}}}\overset{L\pi
^{\ast}}{\longrightarrow}\mathsf{D}^{+}(X,\Lambda_{\bullet})
\]
whose composite we denote by $M\rightsquigarrow\hat{M}$. The essential image
$\mathsf{D}^{+}(X,\Lambda_{\bullet})_{\mathrm{norm}}$ of $L\pi^{\ast}$
consists of the complexes $M$ such that $M\simeq\hat{M}$. Let $\mathsf{D}%
_{c}^{b}(X,\Lambda_{\bullet})$ denote the full subcategory of $\mathsf{D}%
^{+}(X,\Lambda_{\bullet})$ consisting of complexes $M$ such that $M\simeq
\hat{M}$ and $\Lambda_{1}\otimes_{\Lambda_{\bullet}}^{L}M$ lies in
$\mathsf{D}_{c}^{b}(X,\Lambda_{1})$. The canonical functor from $\mathsf{D}%
^{+}(X,\Lambda_{\bullet})_{\mathrm{norm}}$ to Ekedahl's category
$\mathsf{D}^{+}(X,\mathbb{Z}{}_{\ell})$ induces an equivalence of categories%
\[
\mathsf{D}_{c}^{b}(X,\Lambda_{\bullet})\rightarrow\mathsf{D}_{c}%
^{b}(X,\mathbb{Z}{}_{\ell})\text{.}%
\]
See \cite{fargues2009}, 5.15.
\end{plain}

\begin{plain}
\label{a22}Let $k$ be a field, and let $\Gamma$ be its absolute Galois group.
Let $\Mod(\mathbb{Z}{}_{l}\Gamma)$ denote the category of $\mathbb{Z}{}_{l}%
$-modules with a continuous action of $\Gamma$, and let $\mathsf{D}%
(\mathbb{Z}{}_{l}\Gamma)$ denote its derived category. Define $\mathsf{D}%
_{c}^{b}(\mathbb{Z}{}_{l}\Gamma)$ to be the subcategory of $\mathsf{D}%
(\mathbb{Z}{}_{l}\Gamma)$ of bounded complexes with finitely generated
cohomology (as a $\mathbb{Z}{}_{l}$-module). It is a triangulated category
with $t$-structure whose heart is the category of continuous representations
of $\Gamma$ on finitely generated $\mathbb{Z}{}_{l}$-modules. When
$l\neq\mathrm{char}(k)$, the functor sending $(M_{n})_{n{}}\rightsquigarrow
\varprojlim M_{n}(k^{\mathrm{sep}})$ derives to a functor%
\[
\alpha\colon\mathsf{D}_{c}^{b}(k,\mathbb{Z}{}_{l})\rightarrow\mathsf{D}%
_{c}^{b}(\mathbb{Z}{}_{l}\Gamma)\text{.}%
\]

\end{plain}

\begin{plain}
\label{a23}Now let $k$ be a finite field with $q$ elements, and equip
$\Gamma\overset{\textup{{\tiny def}}}{=}\Gal(\bar{k}/k)$ with the topological
generator $\gamma\colon x\mapsto x^{q}$. For $M$ in $\mathsf{D}_{c}%
^{b}(\mathbb{Z}{}_{l}\Gamma)$, we write $\bar{M}$ for $M$ as an object of
$\mathsf{D}_{c}^{b}(\mathbb{Z}{}_{l})$.

Let $M,N\in\mathsf{D}_{c}^{b}(\mathbb{Z}{}_{\ell}\Gamma)$, and let
$P=R\underline{\Hom}(M,N)$. Note that%
\[
P=R\Hom(\bar{M},\bar{N})
\]
regarded as a continuous $\mathbb{Z}{}_{\ell}\Gamma$-module. Let%
\[
f_{j}\colon\Ext^{j}(\bar{M},\bar{N}(r))^{\Gamma}\rightarrow\Ext^{j}(\bar
{M},\bar{N}(r))_{\Gamma}%
\]
be the map induced by the identity map on $\Ext^{j}(\bar{M},\bar{N}(r))$.

Let $[S]$ denote the cardinality of a set $S$. For a homomorphism $f\colon
M\rightarrow N$ of abelian groups, we let $z(f)=[\Ker(f)]/[\Coker(f)]$ when
both cardinalities are finite. On applying Lemma 5.1 of \cite{milneR2013} to
the $\mathbb{Z}{}_{\ell}\Gamma$-module%
\[
H^{j}(P)=\Ext^{j}(\bar{M},\bar{N})\text{,}%
\]
we obtain the following statement:\bquote$z(f_{j})$ is defined if and only if
the minimum polynomial of $\gamma$ on $H^{j}(P(r))_{\mathbb{Q}{}}$ does not
have $q^{r}$ as a multiple root, in which case%
\[
z(f_{j})=\left\vert \prod_{i,\,\,a_{j,i}\neq q^{r}}\left(  1-\frac{a_{j,i}%
}{q^{r}}\right)  \right\vert ,
\]
where $(a_{j,i})_{i}$ is the family of eigenvalues of $\gamma$ acting on
$H^{j}(P)_{\mathbb{Q}{}}$.\equote

\end{plain}

\begin{plain}
\label{a24}By a $\Lambda_{\bullet}\Gamma$-module, we mean an inverse system
\[
M=\left(  M_{0}\leftarrow\cdots\leftarrow M_{m}\leftarrow M_{m+1}%
\leftarrow\cdots\right)
\]
with $M_{m}$ a discrete $\Gamma$-module killed by $l^{m}$. For example,
$\Lambda_{\bullet}$ is the $\Lambda_{\bullet}\Gamma$-module $(\mathbb{Z}%
{}/l^{m}\mathbb{Z}{})_{m}$ with the trivial action of $\Gamma$. Let
\[
F\colon\Mod(\Lambda_{\bullet}\Gamma)\rightarrow\Mod(\mathbb{Z}{}_{l})
\]
denote the functor sending $M=(M_{m})$ to the $\mathbb{Z}{}_{l}$-module
$\varprojlim M_{m}^{\Gamma}$. If $M$ satisfies the Mittag-Leffler condition,
then%
\[
R^{j}F(M)\simeq H_{\mathrm{cts}}^{j}(\Gamma,\varprojlim M_{m})
\]
(cohomology with respect to continuous cocycles).

Let $\mathsf{D}(\Lambda_{\bullet}\Gamma)$ denote the derived category of
complexes of $\Lambda_{\bullet}\Gamma$-modules. Then $F$ derives to a functor
$RF\colon\mathsf{D}^{+}(\Lambda_{\bullet}\Gamma)\rightarrow\mathsf{D}%
^{+}(\mathbb{Z}{}_{l})$. For $X$ in $\mathsf{D}(\Lambda_{\bullet}\Gamma)$,%
\begin{equation}
RF(X)\simeq s(\vec{X}\overset{1-\gamma}{\longrightarrow}\vec{X}%
)=\text{\textrm{Cone}}(1-\gamma)[-1] \label{eq12}%
\end{equation}
where $\vec{X}=(R\varprojlim)(X)$ (see \cite{milneR2013}, \S 5). From
(\ref{eq12}), we get a long exact sequence%
\begin{equation}
\cdots\rightarrow H^{j-1}(\vec{X})\overset{1-\gamma}{\longrightarrow}%
H^{j-1}(\vec{X})\rightarrow R^{j}F(X)\rightarrow H^{j}(\vec{X}%
)\overset{1-\gamma}{\longrightarrow}H^{j}(\vec{X})\rightarrow\cdots
\label{eq13}%
\end{equation}

The canonical generator $\gamma$ of $\Gamma$ defines a canonical element
$\theta_{l}$ in $H_{\mathrm{cts}}^{1}(\Gamma,\mathbb{Z}{}_{l})$, which we can
regard as an element of%
\[
\Ext_{\Lambda_{\bullet}\Gamma}^{1}(\Lambda_{\bullet},\Lambda_{\bullet}%
)\simeq\Hom_{\mathsf{D}(\Lambda_{\bullet}\Gamma)}(\Lambda_{\bullet}%
,\Lambda_{\bullet}[1]).
\]
For each $X$ in $\mathsf{D}^{+}(\Lambda_{\bullet}\Gamma)$, $\theta_{l}$
defines morphisms%
\begin{align*}
\theta_{l}\colon X  &  \rightarrow X[1]\\
R\theta_{l}\colon RF(X)  &  \rightarrow RF(X)[1].
\end{align*}
The second map is described explicitly by the following map of double
complexes:%
\[
\begin{tikzpicture}[baseline=(current bounding box.center), text height=1.5ex, text depth=0.25ex]
\node (a) at (0,0) {$RF(X)$};
\node (b) at (2,0) {};
\node (c) at (4,0) {$\vec{X}$};
\node (d) at (6,0) {$\vec{X}$};
\node (e) [below=of a] {$RF(X)[1]$};
\node (f)[below=of b] {$\vec{X}$};
\node (g)[below=of c] {$\vec{X}$};
\node (h)[below=of d] {};
\node at (2,-2.05) {$\scriptstyle{-1}$};
\node at (4,-2.05) {$\scriptstyle{0}$};
\node at (6,-2.05) {$\scriptstyle{1}$};
\draw[->,font=\scriptsize,>=angle 90]
(a) edge node[right]{$R\theta$} (e)
(c) edge node[right]{$\gamma$} (g)
(c) edge node[above]{$1-\gamma$} (d)
(f) edge node[above]{$1-\gamma$} (g);
\end{tikzpicture}
\]
For all $j$, the following diagram commutes%
\begin{equation}
\begin{tikzcd} R^{j}F(X)\arrow{d}\arrow{r}{d^{j}}& R^{j+1}F(X)\\ H^{j}(\vec{X})\arrow{r}{\id} & H^{j}(\vec{X})\arrow{u} \end{tikzcd} \label{eq14}%
\end{equation}
where $d^{j}=H^{j}(R\theta)$ and the vertical maps are those in (\ref{eq13}).
The sequence%
\begin{equation}
\cdots\rightarrow R^{j-1}F(X)\overset{d^{j-1}}{\longrightarrow}R^{j}%
F(X)\overset{d^{j}}{\longrightarrow}R^{j+1}F(X)\rightarrow\cdots\label{eq15}%
\end{equation}
is a complex because $R\theta\circ R\theta=0$.
\end{plain}

\begin{plain}
\label{a25}Let $M,N\in\mathsf{D}_{c}^{b}(k,\mathbb{Z}{}_{\ell})$. The
bifunctor%
\[
R\Hom\colon\mathsf{D}_{c}^{b}(k,\mathbb{Z}{}_{\ell})^{\mathrm{opp}}%
\times\mathsf{D}_{c}^{b}(\mathbb{Z}{}_{\ell}\Gamma)\rightarrow D^{+}%
(\mathbb{Z}{}_{\ell})
\]
factors canonically through $RF\colon\mathsf{D}^{+}(\mathbb{Z}_{\ell}{}%
\Gamma)\rightarrow\mathsf{D}(\mathbb{Z}{}_{\ell})$:%
\[
\begin{tikzcd}
\mathsf{D}_c^b(k,\mathbb{Z}_{\ell})\arrow{r}\arrow[bend left=20]{rr}{R\Hom(\1,-)}
&\mathsf{D}^{+}(\mathbb{Z}_{\ell}\Gamma)\arrow{r}{RF}
&\mathsf{D}^{+}(\mathbb{Z})\\
\mathsf{D}_c^b(k,\mathbb{Z}_{\ell})^{\mathrm{opp}}\times\mathsf{D}_c^b(k,\mathbb{Z}_{\ell})\arrow{u}{R\underline{\Hom}(-,-)}
\arrow{rru}[swap]{R\Hom(-,-)}
\end{tikzcd}
\]
Hence%
\[
R\Hom(M,N(r))=RF(X)
\]
for a well-defined object $X$ in $\mathsf{D}(\mathbb{Z}{}_{\ell}\Gamma)$. The
sequence (\ref{eq13}) gives us short exact sequences%
\begin{equation}
0\rightarrow\Ext^{j-1}(\bar{M},\bar{N}(r))_{\Gamma}\rightarrow\Ext^{j}%
(M,N(r))\rightarrow\Ext^{j}(\bar{M},\bar{N}(r))^{\Gamma}\rightarrow0
\label{eq16}%
\end{equation}
in which $(-)^{\Gamma}$ and $(-)_{\Gamma}$ denote the kernel and cokernel of
$1-\gamma$. Moreover, (\ref{eq15}) becomes a complex $E(M,N,r)$,%
\[
\quad\cdots\longrightarrow\Ext^{j-1}(M,N(r))\overset{d^{j-1}}{\longrightarrow
}\Ext^{j}(M,N(r))\overset{d^{j}}{\longrightarrow}\Ext^{j+1}%
(M,N(r))\longrightarrow\cdots\text{.}%
\]
This is the unique complex for which the following diagram commutes,%
\begin{equation}
\begin{tikzcd}[column sep=scriptsize] {}&{}&\Ext^j(\bar{M},\bar{N}(r))^{\Gamma}\arrow{r}{f^j} &\Ext^{j}(\bar{M},\bar{N}(r))_{\Gamma}\arrow{d}\\ \cdots\arrow{r}&\Ext^{j-1}(M,N(r))\arrow{d}\arrow{r}{d^{j-1}} &\Ext^{j}(M,N(r))\arrow{r}{d^{j}}\arrow{u} &\Ext^{j+1}(M,N(r))\arrow{r}&\cdots\\ &\Ext^{j-1}(\bar{M},\bar{N}(r))^{\Gamma}\arrow{r}{f^{j-1}} &\Ext^{j-1}(\bar{M},\bar{N}(r))_{\Gamma}\arrow{u} \end{tikzcd} \label{e17}%
\end{equation}
(the vertical maps are those in (\ref{eq14}) and the maps $f^{j}$ are induced
by the identity map).
\end{plain}

\begin{proposition}
\label{a26}Let $M,N\in\mathsf{D}_{c}^{b}(k,\mathbb{Z}{}_{\ell})$, and let
$P=R\underline{\Hom}(M,N)$. Let $r\in\mathbb{Z}{}$, and assume that, for all
$j$, the minimum polynomial of $\gamma$ on $H^{j}(\alpha P)_{\mathbb{Q}%
{}_{\ell}}$ does not have $q^{r}$ as a multiple root.

\begin{enumerate}
\item The groups $\Ext^{j}(M,N(r))$ are finitely generated $\mathbb{Z}{}%
_{\ell}$-modules, and the alternating sum of their ranks is zero.

\item The zeta function $Z(P,t)$ of $P$ has a pole at $t=q^{-r}$ of order%
\[
\rho=\sum(-1)^{j+1}\cdot j\cdot\rank_{\mathbb{Z}{}_{\ell}}\left(
\Ext^{j}(M,N(r))\right)  .
\]

\item The cohomology groups of the complex $\Ext^{\bullet}(M,N(r))$ are
finite, and the alternating product $\chi^{\times}(M,N(r))$ of their orders
satisfies%
\[
\left\vert \lim_{t\rightarrow q^{-r}}Z(P,t)(1-q^{r}t)^{\rho}\right\vert
_{\ell}^{-1}=\chi^{\times}(M,N(r)).
\]

\end{enumerate}
\end{proposition}

\begin{proof}
(a) Note that $\Ext^{j}(\bar{M},\bar{N}(r))=H^{j}(\alpha P(r))$, which is a
finitely generated $\mathbb{Z}{}_{\ell}$-module. From (\ref{eq16}), we see
that%
\[
\rank(\Ext^{j}(M,N(r)))=\rank(H^{j-1}(\alpha P(r))_{\Gamma})+\rank^{j}%
(H^{j}(\alpha P(r))^{\Gamma}).
\]
The hypothesis on the action of the Frobenius element implies that%
\[
H^{j}(\alpha P(r))^{\Gamma}\otimes\mathbb{Q}{}\simeq H^{j}(\alpha
P(r))_{\Gamma}\otimes\mathbb{Q}{}%
\]
for all $j$, and so%
\begin{align*}
&  \sum\nolimits_{j}(-1)^{j}\rank(\Ext^{j}(M,N(r)))\\
&  =\sum\nolimits_{j}(-1)^{j}\left(  \rank(H^{j-1}(\alpha P(r))^{\Gamma
})+\rank(H^{j}(\alpha P(r))^{\Gamma})\right) \\
&  =0
\end{align*}

(b) Let $\rho_{j}$ be the multiplicity of $q^{r}$ as an inverse root of
$P_{j}$. Then%
\[
\rho_{j}=\rank H^{j}(\alpha P(r))^{\Gamma}=\rank H^{j}(\alpha P(r))_{\Gamma
}\text{,}%
\]
and so%
\begin{align*}
\sum\nolimits_{j}(-1)^{j+1}\cdot j\cdot\rank(\Ext^{j}(M,N(r)))  &
=\sum\nolimits_{j}(-1)^{j+1}\cdot j\cdot(\rho_{j-1}+\rho_{j})\\
&  =\sum\nolimits_{j}(-1)^{j+1}\rho_{j}\\
&  =-\rho\text{.}%
\end{align*}

(c) From Lemma 5.2 of \cite{milneR2013} applied to the diagram (\ref{e17}), we
find that%
\[
\chi^{\times}(M,N(r))=\prod\nolimits_{j}z(f^{j})^{(-1)^{j}}\text{.}%
\]
According to (\ref{a23}),%
\[
z(f^{j})=\left\vert \prod_{i,\,\,a_{j,i}\neq q^{r}}\left(  1-\frac{a_{j,i}%
}{q^{r}}\right)  \right\vert _{\ell}~~
\]
where $(a_{j,i})_{i}$ is the family of eigenvalues of $\gamma$ acting on
$H^{j}(\alpha P(r))_{\mathbb{Q}{}_{\ell}}$. Note that%
\[
\prod_{i\,,\,\,\,a_{j,i}\neq q^{r}}\left(  1-\frac{a_{j,l}}{q^{r}}\right)
=\lim_{t\rightarrow q^{-r}}\frac{P_{j}(t)}{(1-q^{r}t)^{\rho_{j}}}\text{,}%
\]
and so%
\[
\chi^{\times}(M,N(r))=\left\vert \lim_{t\rightarrow q^{-r}}Z(M,N,t)\cdot
(1-q^{r}t)^{\rho}\right\vert _{\ell}^{-1}.
\]

\end{proof}

\section{The local conjecture at $p$}

This section is a brief review of Milne and Ramachandran 2013. The definitions
and results reviewed in (\ref{b1}, \ref{b2}, \ref{b3}) are due to Ekedahl,
Illusie, and Raynaud.

\begin{plain}
\label{b1}Let $k$ be a perfect field, and let $W$ be the ring of Witt vectors
over $k$ equipped with its Frobenius automorphism $\sigma$. We let $K$ denote
the field of fractions of $W$. Recall that the Raynaud ring is the graded
algebra $~R=R^{0}\oplus R^{1}=R^{0}[d]$ where $R^{0}$ is the Dieudonn\'{e}
ring $W_{\sigma}[F,V]$ and $d$ (of degree $1$) satisfies $d^{2}=0$, $FdV=d$,
$ad=da$ ($a\in W$).
\end{plain}

\begin{plain}
\label{b2}The graded $R$-modules and homomorphisms of degree $0$ form an
abelian category, whose derived category is denoted by $\mathsf{D}(R)$. There
is a well-defined \textquotedblleft completion\textquotedblright\ functor
$M\rightsquigarrow\hat{M}\colon\mathsf{D}(R)\rightarrow\mathsf{D}(R)$, and an
object $M$ of $\mathsf{D}(R)$ is said to be complete if $M\simeq\hat{M}$. A
coherent complex of graded $R$-modules is a complete complex $M$ in
$\mathsf{D}(R)$ such that $R_{1}\otimes_{R}^{L}M$ is bounded with
finite-dimensional cohomology. Here $R_{1}=R/\left(  VR+dVR\right)  $. The
coherent complexes form a full triangulated subcategory $\mathsf{D}_{c}%
^{b}(R)$ of $\mathsf{D}(R)$.
\end{plain}

\begin{plain}
\label{b3}A complex of graded $R$-modules is often viewed as a bicomplex
$M=M^{\bullet\bullet}$ of $R^{0}$-modules in which the first index corresponds
to the $R$-gradation. Let $F^{\prime}$ act on $M^{i\bullet}$ as $p^{i}F$
(assuming only nonnegative $i$'s occur). The differentials in the bicomplex
commute with $F^{\prime}$, and so the associated simple complex $sM$ is a
complex of $W_{\sigma}[F^{\prime}]$-modules. When $M$ is coherent,
$H^{n}(sM)_{K}$ is an $F$-isocrystal.
\end{plain}

We now take $k=\mathbb{F}{}_{q}$, $q=p^{a}$.

\begin{plain}
\label{b4}We define the zeta function $Z(M,t)$ of an $M$ in $\mathsf{D}%
_{c}^{b}(R)$ to be the alternating product of the characteristic polynomials
of $F^{a}$ acting on the $F$-isocrystals $H^{n}(sM)_{K}$ (see \S 3).
\end{plain}

\begin{plain}
\label{b5}Let $\Gamma$ denote the Galois group of $\bar{k}/k$ equipped with
its generator $x\mapsto x^{q}$. By a $\Lambda_{\bullet}\Gamma$-module, we mean
an inverse system $M=(M_{m})_{m\in\mathbb{N}{}}$ of discrete $\Gamma$-modules
$M_{m}$ killed by $p^{m}$. For example, $\Lambda_{\bullet}$ denotes the
$\Lambda_{\bullet}\Gamma$-module $(\mathbb{Z}{}/p^{m}\mathbb{Z}{})_{m}$. Let
$F$ denote the functor sending a $\Lambda_{\bullet}\Gamma$-module $M$ to the
$\mathbb{Z}{}_{p}$-module $\varprojlim M_{m}^{\Gamma}$. If $M$ satisfies the
Mittag-Leffler condition, then%
\[
R^{j}F(M)\simeq H_{\mathrm{cts}}^{j}(\Gamma,\varprojlim M_{m})
\]
(cohomology with respect to continuous cocycles). Because $\Gamma$ has a
canonical generator, there is a canonical element $\theta_{p}$ in
$H_{\mathrm{cts}}^{1}(\Gamma,\mathbb{Z}{}_{p})$, which we can regard as an
element of%
\[
\Ext_{\Lambda_{\bullet}\Gamma}^{1}(\Lambda_{\bullet},\Lambda_{\bullet}%
)\simeq\Hom_{\mathsf{D}(\Lambda_{\bullet}\Gamma)}(\Lambda_{\bullet}%
,\Lambda_{\bullet}[1]).
\]
Here $\mathsf{D}(\Lambda_{\bullet}\Gamma)$ is the derived category of the
category of $\Lambda_{\bullet}\Gamma$-modules. For each $X$ in $\mathsf{D}%
^{+}(\Lambda_{\bullet}\Gamma)$, $\theta_{p}$ defines morphisms%
\begin{align*}
\theta_{p}\colon X  &  \rightarrow X[1]\\
R\theta_{p}\colon RF(X)  &  \rightarrow RF(X)[1].
\end{align*}

\end{plain}

\begin{plain}
\label{b6} The functor%
\[
R\Hom\colon\mathsf{D}_{c}^{b}(R)^{\mathrm{opp}}\times\mathsf{D}_{c}%
^{b}(R)\rightarrow D(\mathbb{Z}{}_{p})
\]
factors canonically through%
\[
RF\colon\mathsf{D}^{+}(\Lambda_{\bullet}\Gamma)\rightarrow\mathsf{D}%
(\mathbb{Z}{}_{p})\text{.}%
\]
This means that, for each pair $M,N\in\mathsf{D}_{c}^{b}(R)$, there is a
well-defined $X\in$ $\mathsf{D}^{+}(\Lambda_{\bullet}\Gamma)$ such that
\[
R\Hom(M,N(r))=RF(X)\text{.}%
\]
Now $R\theta_{p}$ is a morphism%
\[
R\Hom(M,N(r))\rightarrow R\Hom(M,N(r))[1]
\]
On setting $\Ext^{j}(M,N(r))=H^{j}(R\Hom(M,N(r)))$, we obtain a complex%
\[
E(M,N,r)\colon\quad\cdots\rightarrow\Ext^{j}(M,N(r))\overset{d^{j}%
}{\longrightarrow}\Ext^{j+1}(M,N(r))\rightarrow\cdots\text{.}%
\]

\end{plain}

\begin{proposition}
\label{b7}Let $M,N\in\mathsf{D}_{c}^{b}(R)$, and let $P=R\underline{\Hom}%
(M,N)$. Let $r\in\mathbb{Z}{}$, and assume that $q^{r}$ is not a multiple root
of the minimum polynomial of $F^{a}$ acting on $H^{j}(sP)_{K}$.

\begin{enumerate}
\item The groups $\Ext^{j}(M,N(r))$ are finitely generated $\mathbb{Z}{}_{p}%
$-modules, and the alternating sum of their ranks is zero.

\item The zeta function $Z(P,t)$ has a pole at $t=q^{-r}$ of order%
\[
\rho=\sum(-1)^{j+1}\cdot j\cdot\rank_{\mathbb{Z}{}_{p}}(\Ext^{j}(M,N(r)).
\]

\item The cohomology groups of the complex $E(M,N,r)$ are finite, and the
alternating product $\chi^{\times}(M,N(r))$ of their orders satisfies
\[
\left\vert \lim_{t\rightarrow q^{-r}}Z(M,N,t)\cdot(1-q^{r}t)^{\rho}\right\vert
_{p}^{-1}=\chi^{\times}(M,N(r))\cdot q^{\chi(M,N,r)}%
\]
where%
\[
\chi(M,N,r)=\sum\nolimits_{i<r}(-1)^{i}(r-i)\left(  \sum\nolimits_{j}%
(-1)^{j}h^{i,j}(P)\right)  .
\]

\end{enumerate}
\end{proposition}

\section{How to prove the conjectures}

This section consists of a series of somewhat unrelated subsections.
Throughout, $\mathsf{DM}(k)$ is a triangulated category of motivic complexes
over a perfect field $k$, and $r_{l}$ is an exact realization functor to
$\mathsf{D}_{c}^{b}(k,\mathbb{Z}{}_{l})$ ($l\neq p$) or $\mathsf{D}_{c}%
^{b}(R)$ ($l=p)$. We sometimes write $\mathsf{D}_{c}^{b}(k,\mathbb{Z}{}_{p})$
for $\mathsf{D}_{c}^{b}(R)$.

\subsection{It suffices to construct $\mathsf{DM}(k)$ for $k$ algebraically
closed\label{a}}

Throughout this subsection, all dg categories are pretriangulated. We write
$\mathcal{D}_{c}^{b}(k{},\mathbb{Z}{}_{\ell})$ (resp. $\mathcal{D}_{c}^{b}%
(k{},\mathbb{Z}{}_{p})$) for the natural dg enhancement of $\mathsf{D}_{c}%
^{b}(k,\mathbb{Z}{}_{\ell})$ (resp. $\mathsf{D}_{c}^{b}(R)$). When
$k=\mathbb{F}{}$, each object of $\mathcal{D}_{c}^{b}(k{},\mathbb{Z}{}_{l})$
is equipped with a germ of a Frobenius element (cf. \cite{milne1994}, p.422).
We assume that the same is true of $\mathcal{DM}(\mathbb{F}{})$ and that the
realization functors send germs to germs.

Suppose that we have constructed a dg category $\mathcal{D}{}\mathcal{M}%
{}(\mathbb{F}{})$ and dg realization functors $r_{l}\colon\mathcal{\mathcal{D}%
{}\mathcal{M}{}}(\mathbb{F}{})\rightarrow\mathcal{D}{}_{c}^{b}(\mathbb{F}%
{},\mathbb{Z}{}_{l})$ for each $l$ (including $l=p$). In this subsection we
construct a dg category $\mathcal{D}{}\mathcal{M}{}(\mathbb{F}{}_{q})$ and dg
functors
\begin{align*}
\mathcal{D}{}\mathcal{M}{}(\mathbb{F}{}_{q})  &  \rightarrow
\mathcal{DM(\mathbb{F}{}}),\\
r_{l}\colon\mathcal{\mathcal{D}{}\mathcal{M}{}}(\mathbb{F}{}_{q})  &
\rightarrow\mathcal{D}{}_{c}^{b}(\mathbb{F}_{q}{},\mathbb{Z}{}_{l})
\end{align*}
such that%

\begin{equation}
\begin{tikzcd} \mathcal{DM}(\mathbb{F})\arrow{r}{r_l}& \mathcal{D}_{c}^{b}(\mathbb{F},\mathbb{Z}_{l})\\ \mathcal{DM}(\mathbb{F}_{q})\arrow{r}{r_l}\arrow{u}& \mathcal{D}_{c}^{b}(\mathbb{F}_{q},\mathbb{Z}{}_{l})\arrow{u} \end{tikzcd} \label{eq22}%
\end{equation}
commutes. Moreover, the composite%
\[
\mathcal{DM}(\mathbb{F}{}_{q})\rightarrow\mathcal{\mathcal{D}{}\mathcal{M}{}%
}(\mathbb{F})\xrightarrow{R\Hom(\1,-)}\mathcal{D}{}(\mathbb{Z}{})
\]
factors canonically through $\mathcal{D}{}(\mathbb{Z}{}\Gamma_{0})$:%
\begin{equation}
\begin{tikzcd} \mathcal{DM}(\mathbb{F})\arrow{r}& \mathcal{D}{}(\mathbb{Z}{})\\ \mathcal{DM}(\mathbb{F}_{q})\arrow{r}\arrow{u} & \mathcal{D}(\mathbb{Z}\Gamma_{0})\arrow{u}\text{.}\end{tikzcd} \label{eq23}%
\end{equation}

We define $\mathcal{\mathcal{D}{}\mathcal{M}{}}(\mathbb{F}{}_{q})$ to be the
category whose objects are pairs $(X,\pi_{X})$ consisting of an object $X$ of
$\mathcal{\mathcal{D}{}\mathcal{M}{}}(\mathbb{F}{})$ and a $q$-representative
of the germ of a Frobenius element on $X$. A morphism $(X,\pi_{X}%
)\rightarrow(Y,\pi_{Y})$ is a morphism $X\rightarrow Y$ sending $\pi_{X}$ to
$\pi_{Y}$.

Similarly, let $\mathcal{D}{}_{c}^{b}(\mathbb{F}_{q}{},\mathbb{Z}{}%
_{l})^{\prime}$ be the category whose objects are pairs $(X,\pi_{X})$
consisting of an object $X$ of $\mathcal{D}{}_{c}^{b}(\mathbb{F}{}%
,\mathbb{Z}{}_{l})$ and a $q$-representative of the germ of a Frobenius
element on $X$. The functors%
\begin{align*}
(X,\pi_{X})\rightsquigarrow X  &  \colon\mathcal{DM}(\mathbb{F}{}%
_{q})\rightarrow\mathcal{DM}(\mathbb{F}{})\\
(X,\pi_{X})\rightsquigarrow(r_{l}(X),r_{l}(\pi_{X}))  &  \colon\mathcal{DM}%
(\mathbb{F}{}_{q})\rightarrow\mathcal{D}_{c}^{b}(\mathbb{F}{}_{q},\mathbb{Z}%
{}_{l})^{\prime}%
\end{align*}
clearly make the diagram (\ref{eq22}) commute. The functor
\[
X\rightsquigarrow(\bar{X},\pi_{X})\colon\mathcal{D}{}_{c}^{b}(\mathbb{F}_{q}%
{},\mathbb{Z}{}_{l})\rightarrow\mathcal{D}{}_{c}^{b}(\mathbb{F}_{q}%
{},\mathbb{Z}{}_{l})^{\prime}%
\]
is an equivalence of categories. On choosing a quasi-inverse, we obtain the
functors making (\ref{eq22}) commute.

An object of $\mathsf{D}(\mathbb{Z}{}\Gamma_{0})$ is just an object of
$\mathsf{D}(\mathbb{Z}{})$ together with an action of $\gamma_{0}$, and so the
factorization of $R\Hom(\1,-)\colon\mathsf{DM}(\mathbb{F}_{q}{})\rightarrow
\mathsf{D}(\mathbb{Z}{})$ through $\mathsf{D}(\mathbb{Z}{}\Gamma_{0})$ is obvious.

\begin{proposition}
\label{c9}If $r_{l}\colon\mathcal{\mathcal{D}{}\mathcal{M}{}}(\mathbb{F}%
{})\rightarrow\mathcal{D}{}_{c}^{b}(\mathbb{F}{},\mathbb{Z}{}_{l})$ induces an
equivalence of categories
\[
r_{l}\colon\mathcal{\mathcal{D}{}\mathcal{M}{}}(\mathbb{F},\mathbb{Z}{}%
/l^{n}\mathbb{Z}{}{})\rightarrow\mathcal{D}{}_{c}^{b}(\mathbb{F}{}%
,\mathbb{Z}{}/l^{n}\mathbb{Z}{})
\]
for all $n$, then $r_{l}\colon\mathcal{\mathcal{D}{}\mathcal{M}{}}%
(\mathbb{F}_{q}{})\rightarrow\mathcal{D}{}_{c}^{b}(\mathbb{F}_{q}{}%
,\mathbb{Z}{}_{l})$ induces an equivalence of categories
\[
r_{l}\colon\mathcal{\mathcal{D}{}\mathcal{M}{}}(\mathbb{F}_{q},\mathbb{Z}%
{}/l^{n}\mathbb{Z}{}{})\rightarrow\mathcal{D}{}_{c}^{b}(\mathbb{F}_{q}%
{},\mathbb{Z}{}/l^{n}\mathbb{Z}{})
\]
for all $n$.
\end{proposition}

\begin{proof}
Omitted.
\end{proof}

Here $\mathcal{\mathcal{D}{}\mathcal{M}{}}(k,\mathbb{Z}{}/l^{n}\mathbb{Z}{}%
{})$ is the subcategory of $\mathcal{\mathcal{D}{}\mathcal{M}{}}%
(k,\mathbb{Z}_{l}{}{}{})$ of objects killed by $l^{n}$.

\subsection{Some folklore\label{b}}

Throughout this subsection, $k$ is finitely generated over the prime
field.\footnote{So either $k$ is finite or a finitely generated field
extension of $\mathbb{Q}{}$.} We write $P\rightsquigarrow\bar{P}$ for base
change to the algebraic closure $\bar{k}$ of $k$. Let $\Gamma=\Gal(\bar{k}/k)$.

Let $P$ be an object of $\mathsf{DM}(k)$. Set%
\begin{equation}
\renewcommand{\arraystretch}{1.3}\left\{
\begin{array}
[c]{rcl}%
H_{\mathrm{abs}}^{j}(P,r) & = & \Hom_{\mathsf{DM}(k)}(\1,P[j](r))\\
H^{j}(P,\mathbb{Z}_{l}(r)) & = & \Hom_{\mathsf{D}(k,\mathbb{Z}{}_{l}%
)}(\1,r_{l}(P)[j](r))\\
H^{j}(P,\left(  \mathbb{Z}{}/l^{n}\mathbb{Z}{}\right)  (r)) & = &
\Hom_{\mathsf{D}(k,\mathbb{Z}{}_{l})}(\1,r_{l}(P^{(l^{n})})[j](r))\text{.}%
\end{array}
\right.  \label{eq27}%
\end{equation}
Recall that, for $l=p$, $\mathsf{D}(k,\mathbb{Z}{}_{l}%
)\overset{\textup{{\tiny def}}}{=}\mathsf{D}_{c}^{b}(R)$ and that, for an
object $P$ of a triangulated category, $P^{(l^{n})}$ is the cone on
$l^{n}\colon P\rightarrow P$.

Let $A^{j}(P,r)$ denote the image of the canonical map
\[
H_{\mathrm{abs}}^{j}(P,r)_{\mathbb{Q}{}}\rightarrow V_{l}^{j}%
(P,r)\overset{\textup{{\tiny def}}}{=}\mathbb{Q}_{l}\otimes_{\mathbb{Z}{}_{l}%
}{}H^{j}(\bar{P},\mathbb{Z}_{l}(r)).
\]
Let $P^{\prime}$ be a second object of $\mathsf{DM}(k)$ equipped with a
pairing%
\[
P\otimes^{L}P^{\prime}\rightarrow\1(-d)
\]
for some integer $d$. For example, $P^{\prime}=R\underline{\Hom}(P,\1(-d))$.
Assume that the induced pairings%
\[
\langle\,\,,\,\,\rangle\colon V_{l}^{j}(P,r)\times V_{l}^{2d-j}(P^{\prime
},d-r)\rightarrow\mathbb{Q}{}_{l}%
\]
are nondegenerate, and that $V_{l}^{j}(P,r)$ and $V_{l}^{2d-j}(P^{\prime
},d-r)$ are (noncanonically) isomorphic as $\Gamma$-modules. Elements $a\in
A^{j}(P,r)$ and $a^{\prime}\in A^{2d-j}(P^{\prime},d-r)$ pair to a rational
number $\langle a,a^{\prime}\rangle$ independent of $l$. Let%
\[
N^{j}(P,r)=\{a\in A^{j}(P,r)\mid\langle a,a^{\prime}\rangle=0\text{ for all
}a^{\prime}\in A^{2d-j}(P^{\prime},d-r)\}\text{.}%
\]
There is a canonical map
\begin{equation}
\mathbb{Q}{}_{l}\otimes_{\mathbb{Q}{}}A^{j}(P,r)\rightarrow V_{l}%
^{j}(P,r)^{\Gamma}. \label{eq20a}%
\end{equation}
Consider the following statements.

\begin{description}
\item[$T_{l}^{j}(P,r)$:] The map (\ref{eq20a}) is surjective, i.e.,
$\mathbb{Q}{}_{l}A^{j}(P,r)=V_{l}^{j}(P,r)^{\Gamma}$.

\item[$I_{l}^{j}(P,r)$:] The map (\ref{eq20a}) is injective.

\item[$E_{l}^{j}(P,r)$:] The vector space $N^{j}(P,r)=0$.

\item[$S_{l}^{j}(P,r)$:] The map $V_{l}^{j}(P,r)^{\Gamma}\rightarrow
V^{j}(P,r)_{\Gamma}$ induced by the identity map is bijective.
\end{description}

In the next statement, we abbreviate $T_{l}^{j}(P,r)$ to $T$ and $T_{l}%
^{2d-j}(P^{\prime},d-r)$ to $T^{\prime}$ etc..

\begin{theorem}
\label{c10a}The following statements are equivalent:

\begin{enumerate}
\item dim$_{\mathbb{Q}{}}(A^{j}(P,r)/N^{j}(P,r))=\dim_{\mathbb{Q}{}_{l}}%
V_{l}^{j}(P,r)^{\Gamma};$

\item $T+E;$

\item $T+T^{\prime}+S;$

\item $T+T^{\prime}+E+E^{\prime}+I+I^{\prime}+S+S^{\prime}$.
\end{enumerate}

\noindent Moreover, if $k$ is finite, then these statements are equivalent to:

\begin{enumerate}
\item[\textup{(e)}] the multiplicity of $1$ as a root of the characteristic
polynomial of the Frobenius on $V_{l}^{j}(P,r)^{\Gamma}$ is equal to
$\dim_{\mathbb{Q}{}}(A^{j}(P,r)/N^{j}(P,r))$.
\end{enumerate}
\end{theorem}

\begin{proof}
See \cite{tate1994}, \S 2, and \cite{milne2009}, \S 1.
\end{proof}

We say that the \emph{Tate conjecture holds} for $P$, $j,r,l$ if the
equivalent conditions (a),$\ldots$,(d) of the theorem hold for $P,j,r,l$. We
say that the \emph{Tate conjecture holds} for $P$ and $l$ if the equivalent
conditions of the theorem hold for all quadruples $P,j,r,l$ (fixed $P,l$).

\subsection{Local torsion (Gabber's theorem)\label{c}}

Throughout this subsection, $k$ is a separably closed field of characteristic
$p\neq0$ (not necessarily perfect).

\begin{conjecture}
\label{c4}Let $M,N\in\ob\mathsf{DM}(k)$, and let $j\in\mathbb{Z}{}$. The group
$\Ext^{j}(r_{l}M,r_{l}N)$ is torsion-free for almost all $l$.
\end{conjecture}

\noindent Let $X$ be an algebraic variety over $k$. Then (\ref{c4}) applied to
$\1$ and the motivic complexes attached to $X$ predicts that the \'{e}tale
cohomology groups $H^{j}(X_{\mathrm{et}},\mathbb{Z}{}_{l}(r))$ and $H_{c}%
^{j}(X_{\mathrm{et}},\mathbb{Z}{}_{l}(r))$ are torsion-free for almost all
$l\neq p$.

\begin{plain}
\label{c5}Let $X$ be a smooth projective variety over $k$. In this case Gabber
(1983) \nocite{gabber1983} shows that $H^{j}(X_{\mathrm{et}},\mathbb{Z}{}%
_{l}(r))$ is torsion-free for almost $l\neq p$. A specialization argument
shows that it suffices to prove this in the case that $k=\mathbb{F}{}$. When
$k=\mathbb{F}{}$, Gabber applied \cite{deligne1980} to obtain the statement.
\end{plain}

\begin{plain}
\label{c6}Let $X$ be proper and smooth over $k$. An application of Chow's
lemma and de Jong's alteration theorem shows that there exists a morphism
$\pi\colon X^{\prime}\rightarrow X$ with $X^{\prime}$ projective and smooth
and $\pi$ projective, surjective, and generically finite and \'{e}tale. The
composite
\[
H^{j}(X_{\mathrm{et}},\mathbb{Z}{}_{\ell}(r))\overset{\pi^{\ast}%
}{\longrightarrow}H^{j}(X_{\mathrm{et}}^{\prime},\mathbb{Z}{}_{\ell
}(r))\overset{\pi_{\ast}}{\longrightarrow}H^{j}(X_{\mathrm{et}},\mathbb{Z}%
{}_{\ell}(r))
\]
is multiplication by $\deg(\pi)$. Hence $H^{j}(X_{\mathrm{et}},\mathbb{Z}%
{}_{\ell}(r))$ is torsion-free for almost all $\ell$ (\cite{suh2012}, 1.4).
\end{plain}

We now take $k=\mathbb{F}{}$.

\begin{plain}
\label{c7}Let $X$ be an arbitrary variety over $k$. Then de Jong's alteration
theorem (\cite{dejong1996}) says that there exists an alteration $X^{\prime
}\rightarrow X$ of $X$ such that $X^{\prime}$ is smooth; moreover, $X^{\prime
}$ can be chosen to be the complement of a divisor with strict normal
crossings in some smooth projective variety.\footnote{An alteration is a
dominant proper morphism that preserves the dimension. Let $D$ be a divisor in
a variety $X$, and let $D_{i}\subset D$, $i\in I$, be its irreducible
components (both $D$ and the $D_{i}$ are closed subvarieties of codimension
$1$ in $X$). Then $D$ is a strict normal crossings divisor if (a) every point
$s\in D$ is a smooth point on $X$, (b) for every $J\subset I$, the closed
subscheme $D_{J}=\bigcap_{i\in J}D_{i}$ is a smooth subvariety of codimension
$\#J$ in $V$.}

Therefore, by the argument in (\ref{c6}), we may suppose that $U$ is an open
subvariety of a smooth projective variety $X$ with complement a strict normal
crossings divisor $D$. There is then an exact sequence%
\[
\cdots\rightarrow H_{c}^{i}(U_{\mathrm{et}},\mathbb{Z}{}_{\ell})\rightarrow
H^{i}(X_{\mathrm{et}},\mathbb{Z}{}_{\ell})\rightarrow H^{i}(D_{\mathrm{et}%
},\mathbb{Z}{}_{\ell})\rightarrow H_{c}^{i+1}(U_{\mathrm{et}},\mathbb{Z}%
{}_{\ell})\rightarrow\cdots.
\]
Induction on the dimension of $X$ allows us to suppose that $H^{\ast
}(D_{\mathrm{et}},\mathbb{Z}{}_{\ell})$ is torsion-free for almost all $\ell$,
and so it remains to show that the cokernel of%
\begin{equation}
H^{i}(X_{\mathrm{et}},\mathbb{Z}{}_{\ell})\rightarrow H^{i}(D_{\mathrm{et}%
},\mathbb{Z}{}_{\ell}) \label{eq24}%
\end{equation}
is torsion-free for almost all $\ell$. This will be true if the map
(\ref{eq24}) arises from a map $h^{i}(X)\rightarrow h^{i}(D)$ of motives in a
category that becomes semisimple when tensored with $\mathbb{Q}{}$.
\end{plain}

\begin{plain}
\label{c8}We now consider the general statement. Assume that the triangulated
category $\mathsf{DM}(k)_{\mathbb{Q}{}}$ is semisimple (i.e., every
distinguished triangle splits). This is certainly expected to be true when
$k=\mathbb{F}{}$ (see, for example, \cite{milne1994}, 2.49). Fix $M$, and let%
\[
N^{\prime}\rightarrow N\rightarrow N^{\prime\prime}\rightarrow N^{\prime}[1]
\]
be a distinguished triangle. If (\ref{c4}) is true for two of the pairs
$(M,N^{\prime})$, $(M,N)$, $(M,N^{\prime\prime})$, then it is true for all
three (cf. the last paragraph). A similar statement is true in the first
variable. It follows that if $\mathsf{DM}(k)$ is generated as a triangulated
category by the motives of smooth varieties (which is true for the examples in
\S 7\footnote{See, for example, \cite{bondarko2011}.} and \S 8), then
(\ref{c4}) is true for all $M,N\in\ob\mathsf{DM}(k)$.
\end{plain}

\subsection{Consequences of $r_{l}(M,N)$ being an isomorphism\label{d}}

In the remainder of this section, we require that each object of
$\mathsf{DM}(k)$ has a zeta function (see \S 3), and that the realization
functors $r_{l}$ are such that there is a canonical map of complexes%
\[
r_{l}E(M,N,r)\rightarrow E(r_{l}M,r_{l}N,r)
\]
for $M$, $N$ in $\mathsf{DM}(k)$. More specifically, we require that the
$r_{l}$ have dg-enhancements, and that they map the factorization of
\[
R\Hom\colon\mathsf{DM}(k)^{\mathrm{opp}}\times\mathsf{DM}^{+}(k)\rightarrow
\mathsf{D}(\mathbb{Z}{})
\]
through $RF\colon\mathsf{D}(\mathbb{Z}\Gamma_{0}{})\rightarrow\mathsf{D}%
(\mathbb{Z}{})$ onto the canonical factorizations of%
\begin{align*}
R\Hom  &  \colon\mathsf{D}(k,\mathbb{Z}{}_{\ell})^{\mathrm{opp}}%
\times\mathsf{D}^{+}(k,\mathbb{Z}{}_{\ell})\rightarrow\mathsf{D}(\mathbb{Z}%
{}_{\ell})\\
R\Hom  &  \colon\mathsf{D}_{c}^{b}(R)^{\mathrm{opp}}\times\mathsf{D}_{c}%
^{b}(R)\rightarrow\mathsf{D}(\mathbb{Z}{}_{p})
\end{align*}
through $\mathsf{D}(\mathbb{Z}{}_{l}\Gamma$).

Let $M,N\in\ob\mathsf{DM}(k)$. From the realization maps we get maps%
\begin{align*}
r_{\ell}(M,N)\colon\Hom_{\mathsf{\mathsf{DM}}(k)}(M,N)\otimes\mathbb{Z}%
{}_{\ell}  &  \rightarrow\Hom_{\mathsf{D}_{c}^{b}(\mathbb{Z}{}_{\ell}%
)}(r_{\ell}M,r_{\ell}N)\\
r_{p}(M,N)\colon\Hom_{\mathsf{DM}(k)}(M,N)\otimes\mathbb{Z}{}_{p}  &
\rightarrow\Hom_{\mathsf{D}_{c}^{b}(R)}(r_{p}M,r_{p}N)
\end{align*}

We say that an abelian group $X$ is \emph{pseudofinite} if it is torsion and
its $l$-primary component $X(l)$ is finite for all primes $l$. The order of
such a group is the formal product $\prod l^{n(l)}$ where $l^{n(l)}$ is the
order of $X(l)$. We say that an abelian group $X$ is pseudofinitely generated
if $X_{\mathrm{tors}}$ is pseudofinite and $X/X_{\mathrm{tors}}$ is finitely generated.

\begin{theorem}
\label{d5}Let $M,N\in\mathsf{DM}(k)$, and let $r\in\mathbb{Z}$. Let${}$
$P=R\underline{\Hom}(M,N)$. Assume that $r_{l}(M,N)$ is an isomorphism for all
primes $l$ and that the Frobenius endomorphism of $P$ is semisimple.

\begin{enumerate}
\item The groups $\Ext^{j}(M,N(r))$ are pseudofinitely generated, and the
alternating sum of their ranks is zero.

\item The zeta function $Z(P,t)$ has a zero at $t=q^{-r}$ of order%
\[
\rho=\sum(-1)^{j+1}\cdot j\cdot\rank_{\mathbb{Z}{}}(\Ext^{j}(M,N(r)).
\]

\item The cohomology groups of the complex $\Ext^{\bullet}(M,N(r))$ are
pseudofinite, and the alternating product $\chi^{\times}(M,N(r))$ of their
orders satisfies
\[
\left\vert \lim_{t\rightarrow q^{-r}}Z(P,t)(1-q^{r}t)^{\rho}\right\vert
=\chi^{\times}(M,N(r))\cdot q^{\chi(P,r)}%
\]
where%
\[
\chi(P,r)=\sum_{i,j\,\,(i\leq r)}(-1)^{i+j}(r-i)h^{i,j}(R_{p}(P)).
\]

\end{enumerate}

\noindent If, in addition, $\mathsf{DM}(k)_{\mathbb{Q}{}}$ is semisimple, then
Conjecture \ref{a1} holds for $M$ and $N$.
\end{theorem}

\begin{proof}
The statements (a,b,c) are an immediate consequence of the hypotheses and
Propositions \ref{a26} and \ref{b7}. The final statement follows from
(\ref{c8}).
\end{proof}

\subsection{Consequences of rigidity and the Tate conjecture\label{e}}

Let $k=\mathbb{F}{}_{q}$. We assume that, for all primes $l$, the realization
functor $r_{l}\colon\mathsf{DM}(k)\rightarrow\mathsf{D}_{c}^{b}(k,\mathbb{Z}%
{}_{l})$ defines an equivalence on the subcategories of objects killed by
$l^{n}$. For Voevodsky's category $\mathsf{DM}_{-,\text{et}}^{\mathrm{eff}%
}(k)$ and $l\neq p$, this is the rigidity theorem (\cite{suslinV1996}, \S 4;
\cite{voevodsky2000}, 3.3.3).

\begin{theorem}
\label{d9}Let $P\in\ob\mathsf{DM}(k)$, and let $r\in\mathbb{Z}{}$. Assume that
the Tate conjecture holds for $P$, $r$, and a fixed prime $l$. Then, for all
$i\in\mathbb{Z}{}$,

\begin{enumerate}
\item the first Ulm subgroup $U^{i}$ of $H_{\mathrm{abs}}^{i}(P,r)$ is
uniquely divisible by $l$;

\item the group $H_{\mathrm{abs}}^{i}(P,r)^{\prime}%
\overset{\textup{{\tiny def}}}{=}H_{\mathrm{abs}}^{i}(P,r)/U^{i}$ is finitely
generated modulo torsion, and its $l$-primary subgroup is finite;

\item the map $H_{\mathrm{abs}}^{i}(P,r)^{\prime}\otimes\mathbb{Z}{}%
_{l}\rightarrow H^{i}(P,\mathbb{Z}{}_{l}(r))$ is an isomorphism for all $i$.
\end{enumerate}
\end{theorem}

The proof will occupy the rest of this subsection.

\subsubsection{Preliminaries on abelian groups}

In this subsection, we review some elementary results on abelian groups from
the first section of the appendix to \cite{milne1986v}. An abelian group $N$
is said to be \emph{bounded} if $nN=0$ for some $n\geq1$, and a subgroup $M$
of $N$ is \emph{pure} if $M\cap mN=mM$ for all $m\geq1$.

\begin{lemma}
\label{z1}(a) Every bounded abelian group is a direct sum of cyclic groups.

(b) Every bounded pure subgroup $M$ of an abelian group $N$ is a direct
summand of $N$.
\end{lemma}

\begin{lemma}
\label{z2}Let $M$ be a subgroup of $N,$ and let $l^{n}$ be a prime power. If
$M\cap l^{n}N=0$ and $M$ is maximal among the subgroups with this property,
then $M$ is a direct summand of $N$.
\end{lemma}

Every abelian group $M$ contains a largest divisible subgroup $M_{\mathrm{div}%
}$, which is obviously contained in the first Ulm subgroup of $M$,
$U(M)\overset{\textup{{\tiny def}}}{=}\bigcap_{n\geq1}nM$. Note that
$U(M/U(M))=0$.

\begin{proposition}
\label{z3}If $M/nM$ is finite for all $n\geq1$, then $U(M)=M_{\mathrm{div}}$.
\end{proposition}

\begin{corollary}
\label{z4}If $TM=0$ and all quotients $M/nM$ are finite, then $U(M)$ is
uniquely divisible (= divisible and torsion-free = a $\mathbb{Q}$-vector space).
\end{corollary}

For an abelian group $M$, we let $M_{l}$ denote the completion of $M$ with
respect to the $l$-adic topology. The quotient maps $M\rightarrow M/l^{n}M$
induce an isomorphism $M_{l}\rightarrow\varprojlim_{n}M/l^{n}M$. The kernel of
the map $M\rightarrow M_{l}$ is $\bigcap\nolimits_{n}l^{n}M$.

\begin{lemma}
\label{z5}Let $N$ be a torsion-free abelian group. If $N/lN$ is finite, then
the $l$-adic completion of $N$ is a free finitely generated $\mathbb{Z}{}_{l}$-module.
\end{lemma}

\begin{proposition}
\label{z6}Let $\phi\colon M\times N\rightarrow\mathbb{Z}{}$ be a bi-additive
pairing of abelian groups whose extension $\phi_{l}\colon M_{l}\times
N_{l}\rightarrow\mathbb{Z}{}_{l}$ to the $l$-adic completions has trivial left
kernel. If $N/lN$ is finite and $\bigcap\nolimits_{n}l^{n}M=0$, then $M$ is
free and finitely generated.
\end{proposition}

\subsubsection{Proof of theorem \ref{d9}}

Let $P\in\ob\mathsf{DM}(k)$. We use the notations (\ref{eq27}). From the
factorization%
\[
\begin{tikzcd}
\mathsf{DM}(k)\arrow{r}\arrow[bend left=20]{rr}{R\Hom(\1,-)}
&\mathsf{D}^{+}(\mathbb{Z}\Gamma_{0})\arrow{r}{RF}
&\mathsf{D}^{+}(\mathbb{Z}),
\end{tikzcd}
\]
we get a spectral sequence%
\[
H^{i}(\Gamma_{0},H_{\mathrm{abs}}^{j}(\bar{P},r))\implies H_{\mathrm{abs}%
}^{i+j}(P,r),
\]
and hence exact sequences%
\[
0\rightarrow H_{\mathrm{abs}}^{i-1}(\bar{P},r)_{\Gamma_{0}}\rightarrow
H_{\mathrm{abs}}^{i}(P,r)\rightarrow H_{\mathrm{abs}}^{i}(\bar{P}%
,r)^{\Gamma_{0}}\rightarrow0.
\]

On applying the rigidity theorem to the exact cohomology sequence of the
distinguished triangle%
\[
P\overset{l^{n}}{\longrightarrow}P\longrightarrow P^{(l^{n})}\longrightarrow
P[1]
\]
(see the Notations), we get an exact sequence%
\begin{equation}
0\rightarrow H_{\mathrm{abs}}^{i}(P,r)^{(l^{n})}\rightarrow H^{i}%
(P,(\mathbb{Z}{}/l^{n}\mathbb{Z}{})(r))\rightarrow H_{\mathrm{abs}}%
^{i+1}(P,r)_{l^{n}}\rightarrow0. \label{eq21}%
\end{equation}
Here the subscript $l^{n}$ denotes the kernel of multiplication by $l^{n}$. We
deduce that $H_{\mathrm{abs}}^{i}(P,r)^{(l^{n})}$ and $H_{\mathrm{abs}}%
^{i+1}(P,r)_{l^{n}}$ are finite for all $i$ and $n$. On passing to the inverse
limit, we get an exact sequence%
\[
0\rightarrow H_{\mathrm{abs}}^{i}(P,r)_{l}\rightarrow H^{i}%
(P,\mathbb{\mathbb{Z}{}}_{l}(r)\mathbb{)}{}\rightarrow T_{l}H_{\mathrm{abs}%
}^{i+1}(P,r)\rightarrow0.
\]

We now apply $T_{l}^{i}(P,r)$: the map%
\[
{}H_{\mathrm{abs}}^{i}(P,r)\otimes\mathbb{Q}_{l}\rightarrow H^{i}(\bar
{P},\mathbb{Q}{}_{l}(r))^{\Gamma}%
\]
is surjective. This implies that the cokernel of the map%
\[
H_{\mathrm{abs}}^{i}(P,r)\otimes_{\mathbb{Z}{}}\mathbb{Z}_{l}\rightarrow
H{}^{i}(P,\mathbb{Z}_{l}(r)){}%
\]
is torsion. As the map factors through $H_{\mathrm{abs}}^{i}(P,r)_{l}$, it
follows that $T_{l}H_{\mathrm{abs}}^{i+1}(P,r)=0$ and
\[
H_{\mathrm{abs}}^{i}(P,r)_{l}\simeq H^{i}(P,\mathbb{Z}_{l}{}(r)).
\]
Consider the diagram%
\[
\begin{tikzcd}
H^{i}_{\text{abs}}(P,r)_l
\arrow{r}{\simeq}\arrow{d}{(d^i)_l}
&H^{i}(P,\mathbb{Z}_l(r))\arrow{d}{\delta^i}\\
H^{i+1}_{\text{abs}}(P,r)_l\arrow{r}
&H^{i+1}(P,\mathbb{Z}_l(r)).
\end{tikzcd}
\]
Here the vertical maps are the differentials in $H_{\mathrm{abs}}^{\bullet
}(P,r)$ (see p.\pageref{a2}) and its $l$-analogue. As $\delta^{i}$ has finite
cokernel, so does the bottom arrow, and so $T_{l}H_{\mathrm{abs}}%
^{2i+2}(P,r)=0$. We have now shown that
\[
T_{l}H_{\mathrm{abs}}^{i}(P,r)=0\text{ for all }i
\]
and so
\[
\left\{  \renewcommand{\arraystretch}{1.5}
\begin{array}
[c]{l}%
H_{\mathrm{abs}}^{i}(P,r)_{l}\simeq H^{i}(P,\mathbb{Z}_{l}(r))\\
U(H_{\mathrm{abs}}^{i}(P,r))\text{ is uniquely }l\text{-divisible}%
\end{array}
\right.  \text{ for all }i.
\]
We have now proved (a) of the theorem, and we have proved (b) and (c) except
that each group $H_{\mathrm{abs}}^{i}(P,r)^{\prime}$ has been replaced by its
$l$-adic completion. It remains to prove that $H_{\mathrm{abs}}^{i}%
(P,r)^{\prime}$ is finitely generated for all $i$ (for then $H_{\mathrm{abs}%
}^{i}(P,r)_{l}\simeq H_{\mathrm{abs}}^{i}(P,r)^{\prime}\otimes\mathbb{Z}{}%
_{l}$).

The maps $H_{\mathrm{abs}}^{i}(P,\mathbb{Z}{}(r))^{\prime}\rightarrow
H_{\mathrm{abs}}^{i}(P,\mathbb{Z}{}(r))_{l}$ are injective, and so
$H_{\mathrm{abs}}^{i}(P,\mathbb{Z}{}(r))^{\prime}$ is finite unless $1$ is an
eigenvalue of the Frobenius map on $H^{i}(\bar{P},\mathbb{Z}{}_{l}(r))$ or
$H^{i-1}(\bar{P},\mathbb{Z}{}_{l}(r))$.

We next show that the group $H_{\mathrm{abs}}^{i}(P,\mathbb{Z}{}(r))^{\prime}$
is finitely generated. There is a commutative diagram%
\[
\begin{tikzpicture}[text height=1.5ex, text depth=0.25ex]
\node (a) at (0,0) {$H^{i}_{\text{abs}}(P,r)^{\prime}/\mathrm{tors}$};
\node (t) at (2.2,0) {$\times$};
\node (b) at (5,0) {$H^{2d-i+1}_{\text{abs}}(P,d-r)^{\prime}/\mathrm{tors}$};
\node (c) at (9,0) {$\mathbb{Z}$};
\node (d) [below=of a] {$H^{i}(P,\mathbb{Z}_{l}{}(r))/\mathrm{tors}$};
\node [below=of t] {$\times$};
\node (e) [below=of b] {$H^{2d-i+1}(P,\mathbb{Z}_{l}{}(d-r))/\mathrm{tors}$};
\node (f) [below=of c] {$\mathbb{Z}_{l}$};
\draw[->,font=\scriptsize,>=angle 90]
(a) edge (d)
(b) edge (e)
(c) edge (f)
(b) edge (c)
(e) edge (f);
\end{tikzpicture}
\]
to which we wish to apply (\ref{z6}). The bottom pairing is nondegenerate,
$U(H_{\text{abs}}^{i}(P,r)^{\prime})=0$ (the quotient of a group by its first
Ulm group has trivial first Ulm group), and the group $H_{\mathrm{abs}%
}^{2d-i+1}(P,\mathbb{Z}{}(d-r))^{(l)}$ is finite. \noindent Therefore
(\ref{z6}) shows that $H_{\text{abs}}^{i}(P,\mathbb{Z}{}(r))^{\prime
}/\mathrm{tors}$ is finitely generated. Because $U(H_{\text{abs}}%
^{i}(P,\mathbb{Z}(r))^{\prime})=0$, the torsion subgroup of $H_{\mathrm{abs}%
}^{i}(P,\mathbb{Z}{}(r))^{\prime}$ injects into the torsion subgroup of
$H^{i}(P,\mathbb{Z}_{l}(r))$, which is finite. Hence $H_{\mathrm{abs}}%
^{i}(P,r)^{\prime}$ is finitely generated modulo prime-to-$l$ torsion.

\subsection{The bijectivity of $r_{l}$\label{f}}

Let $k=\mathbb{F}{}_{q}$. Let $P$ be an object of $\mathsf{DM}(k)$ such that
the Tate conjecture is true for $P$ and all $i,r,l$. According to Theorem
\ref{d9}, the map%
\[
\left(  H_{\mathrm{abs}}^{i}(P,r)/U^{i}\right)  \otimes_{\mathbb{Z}{}%
}\mathbb{Z}{}_{l}\rightarrow H_{\mathrm{et}}^{i}(P,\mathbb{Z}{}_{l}(r))
\]
is an isomorphism for all $i,r,l$. Here $U^{i}$ is the first Ulm subgroup of
$H_{\mathrm{abs}}^{i}(P,r)$, which is uniquely divisible (hence a
$\mathbb{Q}{}$-vector space). Moreover, $H_{\mathrm{abs}}^{i}(P,r)^{\prime
}\overset{\textup{{\tiny def}}}{=}H_{\mathrm{abs}}^{i}(P,r)/U^{i}$ is finitely
generated and its $l$-primary subgroup is finite for all $l$. Therefore the
group $\mathrm{Tor}_{1}^{\mathbb{Z}{}}(H_{\mathrm{abs}}^{i}(P,r)^{\prime
},\mathbb{Z}{}_{l})$ is torsion, and so the sequence%
\[
0\rightarrow U^{i}\otimes\mathbb{Z}{}_{l}\rightarrow H_{\mathrm{abs}}%
^{i}(P,r)\otimes\mathbb{Z}{}_{l}\rightarrow H_{\mathrm{abs}}^{i}(P,r)^{\prime
}\rightarrow0
\]
is exact (because $\mathbb{Q}{}\otimes\mathbb{Z}{}_{l}\simeq\mathbb{Q}{}_{l}$
is torsion-free). To show that $U^{i}=0$, it suffices to show that
\begin{equation}
r_{l}\colon(\1,P)\colon H_{\mathrm{abs}}^{i}(P,r)\otimes\mathbb{Z}{}%
_{l}\rightarrow H_{\mathrm{et}}^{i}(P,\mathbb{Z}{}_{l}(r)) \label{eq25}%
\end{equation}
is injective \textit{for a single }$l$.

When $P$ is the motivic complex with compact support of an algebraic variety
$X$, we write $X$ for $P$.

\begin{lemma}
\label{d10}Suppose that the map%
\begin{equation}
r_{l}(\1,X)\colon H_{\mathrm{abs}}^{\ast}(X,r)\otimes\mathbb{Z}{}%
_{l}\rightarrow H^{\ast}(X,\mathbb{Z}{}_{l}(r)) \label{eq26}%
\end{equation}
is bijective for almost all primes $l$ when $X$ is smooth and projective. Then
this is true for all varieties $X$.
\end{lemma}

\begin{proof}
From an alteration $\pi\colon X^{\prime}\rightarrow X$ we get a commutative
diagram%
\[
\begin{tikzcd}
H_{\mathrm{abs}}^{\ast}(X,r)\otimes\mathbb{Z}_{l}
\arrow{r}{\pi^{\ast}}\arrow{d}
&H_{\mathrm{abs}}^{\ast}(X^{\prime},r)\otimes\mathbb{Z}{}_{l}%
\arrow{r}{\pi_{\ast}}\arrow{d}
&H_{\mathrm{abs}}^{\ast}(X,r)\otimes\mathbb{Z}_{l}\arrow{d}\\
H_{\mathrm{et}}^{\ast}(X,\mathbb{Z}{}_{l}(r))
\arrow{r}{\pi^{\ast}}
&H_{\mathrm{et}}^{\ast}(X,\mathbb{Z}{}_{l}(r))
\arrow{r}{\pi_{\ast}}
&H_{\mathrm{et}}^{\ast}(X,\mathbb{Z}{}_{l}(r)).
\end{tikzcd}
\]
in which both composites $\pi_{\ast}\circ\pi^{\ast}$ are multiplication by
$\deg(\pi)$, and are therefore isomorphisms for almost all primes $l$.
Therefore the statement is true for $X$ if it is true for $X^{\prime}$. Now de
Jong's alteration theorem (\cite{dejong1996}) allow us to suppose that $X$ is
smooth and is the complement of a divisor $D$ with strict normal crossings in
a smooth projective variety $Y$. Induction on the dimension of $X$ allows us
to assume the statement for $D$, and a five-lemma argument using the exact
cohomology sequence%
\[
\cdots\rightarrow H_{c}^{i}(X)\rightarrow H^{i}(Y)\rightarrow H^{i}%
(D)\rightarrow
\]
proves it for $X$.
\end{proof}

\begin{aside}
\label{d10a}Fix a prime $l\neq p$. Gabber's improvement of de Jong's theorem
(see \cite{bondarko2011}, 1.2.1) allows one to assume in the above proof that
the degree of $\pi$ is prime to $l$. Therefore we obtain the following
stronger result: if $r_{l}(\1,X)$ is an isomorphism for all smooth projective
varieties, then it is an isomorphism for all varieties (fixed $l\neq p$). Of
course, with resolution, this would be true also for $l=p$.
\end{aside}

\begin{lemma}
\label{d11}Let $X$ be a smooth projective variety over $\mathbb{F}{}_{q}$ that
satisfies the Tate conjecture. If the ideal of $l$-homologically trivial
correspondences in $\mathrm{CH}^{\dim X}(X\times X)_{\mathbb{Q}{}}$ is nil,
then $r_{l}(\1,X)$ is bijective.
\end{lemma}

\begin{proof}
See, for example, the appendix to \cite{milne1986v}.
\end{proof}

\begin{aside}
\label{d8}For a smooth projective algebraic variety $X$ whose Chow motive is
finite-dimensional, the ideal of $l$-homologically trivial correspondences in
$\mathrm{CH}^{\dim X}(X\times X)_{\mathbb{Q}{}}$ is nil for all prime $l$
(Kimura). It is conjectured (Kimura and O'Sullivan) that the Chow motives of
algebraic varieties are always finite-dimensional, and this is known for those
in the category generated by the motives of abelian varieties. On the other
hand, Beilinson has conjectured that, over finite fields, rational equivalence
with $\mathbb{Q}{}$-coefficients coincides with with numerical equivalence,
which implies that the ideal in question is always null (not merely nil).
\end{aside}

\begin{theorem}
\label{d6}Assume that, for each smooth projective variety $X$ over
$\mathbb{F}{}_{q}$, the Tate conjecture holds for $X$ and, for some prime $l$,
the ideal of $l$-homologically trivial correspondences in $\mathrm{CH}^{\dim
X}(X\times X)_{\mathbb{Q}{}}$ is nil. If $\mathsf{DM}(k)$ is generated as a
triangulated category by the motives of smooth algebraic varieties, then the
maps $r_{l}(M,N)$ are isomorphisms for all $l$ (except possibly $p$).
\end{theorem}

\begin{proof}
The hypotheses imply that $r_{l}(\1,X)$ is an isomorphism for all $l$ if $X$
is smooth and projective (see \ref{d11} and the discussion above). Therefore
$r_{l}(\1,X)$ is an isomorphism for all $l$ (except possibly $p$) if $X$ is
smooth (apply \ref{d10}, \ref{d10a}). If $r_{l}(\1,P)$ is an isomorphism for
two of the terms in a distinguished triangle, then it is an isomorphism for
all three of the terms. Because $\mathsf{DM}(k)$ is generated by the motives
of smooth algebraic varieties, we deduce that $r_{l}(\1,P)$ is an isomorphism
all $l$ (except possibly $p$). This proves the theorem for $(\1,P)$, and it
can be deduced for $(M,N)$ by taking $P=R\underline{\Hom}(M,N)$.
\end{proof}

The exception at $p$ in the theorem can be removed under either of the
following two hypotheses.

(a) \textit{Resolution of singularities. }This allows us to drop the condition
$l\neq p$ in (\ref{d10a}), and hence in the rest of the proof.

(b) \textit{Numerical equivalence coincides with }$l$\textit{-homological
equivalence for }$l=p$\textit{ and at least one other prime }$l_{1}$\textit{
(for smooth projective varieties; $\mathbb{Q}$-coefficients). }Let $X$ be a
smooth variety. If $r_{l}(\1,X)$ is an isomorphism for all $l\neq p$, then the
Tate conjecture holds for $X$ and all $l\neq p$. Under our hypothesis on the
equivalence relations, Condition (e) of (\ref{c10a}) holds for $l_{1}$ if and
only if it holds for $p$ (cf. \cite{katz1994}, p.28). Therefore the Tate
conjecture holds for $X$ and $p$, which in which in turn implies that
$r_{p}(\1,X)$ is an isomorphism (\ref{d9} et seq.). We now know that
$r_{l}(\1,X)$ is an isomorphism for all smooth varieties $X$ and all primes
$l$, and so the proof can be completed as before.

\begin{theorem}
\label{d7}Assume that the Tate and Beilinson conjectures hold for all smooth
projective varieties over $\mathbb{F}{}_{q}$. Then Theorem \ref{d5} holds for
all $M,N$ in $\mathsf{DM}(k)$. If in addition, $\mathsf{\mathsf{DM}%
}(k)_{\mathbb{Q}{}}$ is semisimple, then Conjecture \ref{a1} holds for all
$M,N$ in $\mathsf{DM}(k)$.
\end{theorem}

\begin{proof}
Immediate from the above.
\end{proof}

\section{Motivic complexes following Voevodsky}

Let $k$ be a perfect field of characteristic $p>0$, let $W=W(k)$ be the ring
of Witt vectors over $k$, and let $K=W\otimes{\mathbb{Q}}$ be its fraction
field. Let $\bar{k}$ be an algebraic closure of $k$, and let $\Gamma$ the
Galois group of $\bar{k}$ over $k$. Recall that an $F$-isocrystal is a
$K_{\sigma}[F]$-module that is finite-dimensional as a $K$-vector space and
such that $F$ is bijective.

At present, it is not known how to construct (in a natural way, i.e., without
using realizations) a triangulated category of motivic complexes $k$ equipped
with an integral $p$-adic cohomological realization functor to $\mathsf{D}%
_{c}^{b}(R)$. We propose constructing such a category as a dg fibred product
of Voevodsky's category with the category $\mathsf{D}_{c}^{b}(R)$.

In our earlier article \cite{milneR2004}, we constructed the abelian category
of integral mixed motives as a fibred product. Our construction here is
similar in spirit, but takes place at the level of dg-categories.

There are roughly three steps:

\begin{enumerate}
\item definition of a dg crystalline/rigid realization ($k$ algebraically closed);

\item formation of the fibred product category at the dg level and comparison
of the Hom's in the various categories ($k$ algebraically closed);

\item passage to a finite base field and the Weil \'{e}tale cohomology groups.
\end{enumerate}

\subsection*{Various definitions of $\mathsf{DM}^{\text{eff}}_{\text{gm}}(k,
\mathbb{Q})$}

Voevodsky's category $\mathsf{DM}_{\mathrm{et}}^{\mathrm{eff}}%
(k,\mathbb{\mathbb{Z}{}})$ of effective \'{e}tale motives over $k$ has no
$p$-torsion and is $\mathbb{Z}{}[p^{-1}]$-linear. Thus%
\[
\mathsf{DM}_{\mathrm{et}}^{\mathrm{eff}}(k,\mathbb{\mathbb{Z}{}}%
)=\mathsf{DM}_{\mathrm{et}}^{\mathrm{eff}}(k,\mathbb{\mathbb{Z}{}}[p^{-1}]).
\]

With $\mathbb{Q}{}$-coefficients, there is a canonical equivalence of
categories%
\[
\mathsf{DM}_{\mathrm{Nis}}^{\mathrm{eff}}(k,\mathbb{\mathbb{Q}})\rightarrow
\mathsf{DM}_{\mathrm{et}}^{\mathrm{eff}}(k,\mathbb{\mathbb{Q}})
\]
between the triangulated categories of effective Voevodsky motives and
effective \'{e}tale motives with $\mathbb{Q}{}$-coefficients
(\cite{mazzaVW2006}, 14.30, p.118). In fact, the various definitions of the
triangulated category of motivic complexes all give the same answer for
rational coefficients. In particular, the following categories are all
canonically equivalent (\cite{deglise2013}):

\begin{itemize}
\item $\mathsf{DM}_{\mathrm{et}}(k,\mathbb{Q})$ of \'{e}tale motives (Voevodsky)

\item $\mathsf{DM}_{\mathrm{Nis}}(k,\mathbb{Q})$ of Voevodsky sheaves in the
Nisnevich motives (Voevodsky)

\item $\mathsf{DA}_{\mathrm{et}}(k,\mathbb{Q})$ via motivic homotopy (Morel, Ayoub)

\item $\mathsf{DM}_{\mathrm{h}}(k,\mathbb{Q})$ of $h$-topology motives (Voevodsky).
\end{itemize}

The first two involve presheaves with transfers, but the third one does not.
For a definition of $\mathsf{DA}_{\mathrm{et}}^{\mathrm{eff}}(k,\mathbb{Q})$,
see \cite{ayoub2013}, begining of Section 3, just before 3.1. Ayoub (ibid.,
B.1) has provided a canonical equivalence%
\[
\mathsf{DA}_{\mathrm{et}}^{\mathrm{eff}}(k,\mathbb{Q})\rightarrow
\mathsf{DM}_{\mathrm{et}}^{\mathrm{eff}}(k,\mathbb{Q})\text{.}%
\]
For details on the other categories, see \cite{cisinskiD2012b}

The integral versions of these categories, other than the second, are all
$\mathbb{Z}[p^{-1}]$-linear, i.e., they do not have any $p$-torsion. For this
reason, we write $\mathsf{\mathsf{\mathsf{DM}}}_{\mathrm{et}}(k,\mathbb{Z}%
{}[p^{-1}])$ for the integral version of $\mathsf{\mathsf{\mathsf{\mathsf{DM}%
}}}_{\mathrm{et}}(k,\mathbb{Q}{})$. The $p$-part of the integral Nisnevich
category remains mysterious; no connection, as yet, has been established
between the Nisnevich category and the de Rham-Witt complex.

\subsection*{Glossary of categories}

\begin{itemize}
\item $\mathrm{Ho}(\mathcal{C})$ the homotopy category of a dg-category
$\mathcal{C}{}$. When $\mathcal{C}{}$ is pretriangulated, $\mathrm{Ho}%
(\mathcal{C})$ is a triangulated category.

\item $\mathsf{DM}_{\text{et}}^{\text{eff}}(k,\mathbb{\mathbb{Z}{}}[p^{-1}])$
Voevodsky's triangulated category of effective \'{e}tale motives over $k$.

\item $\mathsf{D}_{c}^{b}(R)$ the triangulated category of coherent complexes
of graded modules over the Raynaud ring $R$ (see \S 5).

\item $\mathsf{D}_{c}^{b}(K_{\sigma}[F])$ the bounded derived category of
complexes of $K_{\sigma}[F]$-modules whose cohomology groups are $F$-isocrystals.

\item $\mathcal{DM}_{\text{et}}^{\text{eff}}(k,\mathbb{Z}{}[p^{-1}])$ the
natural dg enhancement of $\mathsf{\mathsf{\mathsf{DM}}}_{\mathrm{et}%
}^{\mathrm{eff}}(k,\mathbb{Z}{}[p^{=1}])$ (\cite{beilinsonV2008}).

\item $\mathcal{D}{}(B)$ the derived dg-category of an exact category $B$;
this is the dg quotient of the dg category $C(B)$ of unbounded complexes by
the subcategory of acyclic ones (\cite{drinfeld2004}, \cite{tabuada2010q}).

\item $\mathcal{D}_{c}^{b}(R)$ the natural dg enhancement of $\mathsf{D}%
_{c}^{b}(R)$.

\item $\mathcal{D}_{c}^{b}(K_{\sigma}[F])$ the natural dg enhancement of
$\mathsf{D}_{c}^{b}(K_{\sigma}[F])$.

\item $\mathcal{DM}(k,\mathbb{Z})$ the dg-category of integral motivic
complexes over $k$ (constructed below).
\end{itemize}

\subsection*{The dg-enhancement of Voevodsky's category}

For every $\mathbb{Q}{}$-algebra $A$, there is a canonical equivalence of
categories
\begin{equation}
\sigma\colon\mathsf{DM}_{\mathrm{Nis}}^{\mathrm{eff}}(k,A)\rightarrow
\mathsf{DM}_{\mathrm{et}}^{\mathrm{eff}}(k,A) \label{eq30}%
\end{equation}
between the triangulated categories of effective Voevodsky motives and
effective \'{e}tale motives with coefficients in $A$ (\cite{mazzaVW2006},
14.30, p.118). Voevodsky's category $\mathsf{DM}_{\text{Nis}}^{\text{eff}%
}(k,\mathbb{Q})$ of effective geometric motives over $k$ admits a crystalline
realization functor (homology) (\cite{cisinskiD2012}). It is a mixed Weil
cohomology theory in the terminology of Cisinski and D\'{e}glise. When
composed with a quasi-inverse of $\sigma$, it gives a crystalline realization
of $\mathsf{DM}_{\mathrm{et}}^{\mathrm{eff}}(k,\mathbb{\mathbb{Q}})$.

We want to lift this to a dg-realization, i.e. a dg-quasifunctor%
\[
\mathcal{DM}_{\text{et}}^{\text{eff}}(k,\mathbb{Z}[p^{-1}])\rightarrow
\mathcal{D}(K_{\sigma}[F])
\]
where $\mathcal{DM}_{\text{et}}^{\text{eff}}(k,\mathbb{Z}[p^{-1}])$ is the
dg-enhancement of $\mathsf{DM}_{\text{et}}^{\text{eff}}(k,\mathbb{Z}[p^{-1}])$
constructed in \cite{beilinsonV2008}.

We shall freely use \cite{beilinsonV2008} and \cite{volog2012}.

In particular, Vologodsky (2012, Theorem 2, p.384; also the start of Section
2) has provided a very convenient criterion for constructing dg-quasifunctors
on $\mathcal{DM}_{\text{et}}^{\text{eff}}(k,\mathbb{Z}[p^{-1}])$. For rational
coefficients, i.e., for quasi-functors on $\mathcal{DM}_{\text{et}%
}^{\text{eff}}(k,\mathbb{Q})$, this criterion states\footnote{Vologodsky
informed us that, by results of Ayoub, the $h$-topology in the criterion can
be replaced by \'{e}tale topology.}

\begin{quote}
A dg-realization $\mathcal{DM}_{\text{et}}^{\text{eff}}(k,\mathbb{Q}%
)\rightarrow\mathcal{C}$ into a cocomplete compactly generated dg-category
$\mathcal{C}$ is just an \textquotedblleft ordinary\textquotedblright\ functor
from the category of smooth connected schemes to $\mathcal{C}$ which is
$\mathbb{A}^{1}$-homotopy invariant and satisfies the descent property for
Voevodsky's $h$-topology.
\end{quote}

\noindent In other words, $\mathcal{DM}_{\text{et}}^{\text{eff}}%
(k,\mathbb{Q})$ is universal for dg-realizations of motives.

The canonical equivalence $\sigma$ admits a dg-enhancement: the canonical
dg-quasifunctor on the dg-enhancements $\mathcal{DM}_{\mathrm{Nis}%
}^{\mathrm{eff}}(k,A)\rightarrow\mathcal{DM}_{\mathrm{et}}^{\mathrm{eff}%
}(k,A)$ is a homotopy equivalence for every $\mathbb{Q}{}$-algebra $A$
(\cite{volog2012}, p.380, following Remark 2.6, ).

\begin{remark}
\label{v1}The construction of the category $\mathsf{DM}_{\mathrm{et}%
}^{\mathrm{eff}}(k,\mathbb{Q}{})$ is based on $\mathbb{A}^{1}$-homotopy
invariant \'{e}tale presheaves with transfer on the category of all smooth
schemes over $k$. The category $\mathsf{DA}_{\mathrm{et}}^{\mathrm{eff}%
}(k,\mathbb{Q}{})$ in Ayoub's theorem is a variant \textquotedblleft sans
transfers\textquotedblright. The equivalence in Ayoub's theorem provides a
variant of Theorem 2.8(c) of \cite{volog2012}. Because of this, one does not
have to check for good properties with respect to transfers in constructing dg-realizations.
\end{remark}

\subsection*{The dg-realization of rigid homology}

We assume that $k$ is algebraically closed. Then the category of
$F$-isocrystals is semisimple, and so every object of $\mathcal{D}{}%
(K_{\sigma}[F])$ is isomorphic to its homology (viewed as a complex with
trivial differentials).

\begin{lemma}
\label{v2}The category $C(K_{\sigma}[F])$ of unbounded complexes of
$K_{\sigma}[F]$-modules is a pre-triangulated cocomplete compactly generated
dg-category over $\mathbb{Q}_{p}$.
\end{lemma}

\begin{proof}
Apply \cite{beilinsonV2008}, Example before 1.5.5, p.1718, which shows that
for every abelian category $\mathsf{A}$, the category of complexes
$C(\mathsf{A})$ is pretriangulated.
\end{proof}

Our task is to define the dg-quasifunctor
\[
\mathrm{Crys}\colon\mathcal{DM}_{\text{et}}^{\text{eff}}(k,\mathbb{Q}%
)\rightarrow\mathcal{D}(K_{\sigma}[F]).
\]
By Vologodsky's criterion and Ayoub's theorem (\cite{ayoub2013}, B.1) it
suffices to

\begin{enumerate}
\item construct a functor from smooth connected schemes over $k$ to
$C(K_{\sigma}[F])$ (hence to $\mathcal{D}(K_{\sigma}[F])$) that is
$\mathbb{A}^{1}$-homotopy invariant and satisfies \'{e}tale descent, and

\item check that the image lies in the dg subcategory $\mathcal{D}_{c}%
^{b}(K_{\sigma}[F])$) of $\mathcal{D}(K_{\sigma}[F])$).
\end{enumerate}

N. Tsuzuki has proved proper cohomological descent for rigid cohomology and,
together with B. Chiaroletto, \'{e}tale descent for rigid cohomology
(\cite{tsuzuki2003}, \cite{chiaT2007}). It is known that rigid cohomology is
$\mathbb{A}{}^{1}$-homotopy invariant.\footnote{It satisfies the K\"{u}nneth
formula, agrees with Monsky-Washnitzer cohomology on smooth affine varieties,
and the Monsky-Washnitzer cohomology of affine space is trivial except in
degree $0$ (see, for example, \cite{vanderput1986}).}

For a smooth scheme $X$ over $k$, we define $\mathrm{Crys}(X)$ to be Besser's
rigid complex $\mathbb{R}{}\Gamma(X/K)$ (\cite{besser2000}, 4.9, 4.13), which
is a canonical functorial complex of $K$-vector spaces that computes the rigid
cohomology of $X$. It is compatible with base change $k\rightarrow k^{\prime}$
and is endowed with a Frobenius map (ibid. Proposition 4.21, Corollary 4.22).
Therefore $\mathrm{Crys}(X)$ is an object of $C{}(K_{\sigma}[F])$. The
assignment $X\rightsquigarrow\mathrm{Crys}(X)$ is a functor; it satisfies
\'{e}tale descent and $\mathbb{A}{}^{1}$-homotopy invariance because rigid
cohomology satisfies \'{e}tale descent and $\mathbb{A}{}^{1}$-homotopy invariance.

In summary, there exists a dg-crystalline realization functor%
\[
\mathrm{Crys}\colon\mathcal{DM}_{\mathrm{et}}^{\mathrm{eff}}(k,\mathbb{Z}%
[p^{-1}])\rightarrow\mathcal{D}{}(K_{\sigma}[F]).
\]

\subsubsection{Notes}

\begin{plain}
\label{v3}The functor $\mathrm{Crys}$ becomes covariant if we take
\textrm{Crys}$(X)$ to be the dual of Besser's complex.
\end{plain}

\begin{plain}
\label{v4}On $\mathcal{DM}_{\mathrm{et}}^{\mathrm{eff}}(k,\mathbb{Z}[p^{-1}])$
this becomes the Borel-Moore rigid homology functor.
\end{plain}

\begin{plain}
\label{v5}As the Besser complex is compatible with base change, \textrm{Crys}%
$(X)$ for any smooth variety $X$ over a finite field $k$ is a complex of
$K_{\sigma}[F]$-modules with an action of $\Gamma_{0}$.
\end{plain}

\begin{plain}
\label{v6}Another method of obtaining \textrm{Crys} is to use the
overconvergent site of B. Le Stum (\cite{lestum2011}).
\end{plain}

\begin{plain}
\label{v7}There is an alternative construction of \textrm{Crys}. The
restriction functor from \'{e}tale sheaves on smooth schemes over $k$ to
\'{e}tale sheaves on smooth affine schemes over $k$ is an equivalence of
categories. Therefore, it suffices to define \textrm{Crys} on smooth affine
schemes. Here we can take the Monsky-Washnitzer complex of a smooth affine
variety. It is known that the cohomology of this complex computes the rigid
cohomology of the variety.
\end{plain}

\subsection*{Homotopy fibred products of dg-categories}

In this subsection, we review the definition of homotopy fibred products of dg
categories (\cite{drinfeld2004}, Section 15, Appendix IV; \cite{tabuada2010},
Chapter 3; \cite{benbassatB2012}, Section 4).

Let $\mathcal{B}{}$ be a dg category. Given objects $x,y$ of $\mathcal{B}$, we
write $\Hom_{\mathcal{B}{}}^{\bullet}(x,y)$ for the $\mathbb{Z}$-graded
complex of morphisms from $x$ to $y$. For the homotopy category $\mathrm{Ho}%
\mathcal{B}$ of $\mathcal{B}{}$,%
\[
\Hom_{\mathrm{Ho}\mathcal{B}{}}(x,y)=H^{0}(\Hom_{\mathcal{B}{}}^{\bullet
}(x,y)).
\]

Consider a diagram of dg-categories and dg-functors:%
\[
\begin{tikzcd}
{}& \mathcal{B}\arrow{d}{G}\\
\mathcal{C}\arrow{r}{L}&\mathcal{D}.
\end{tikzcd}
\]
The homotopy fibred product $\mathcal{C}\times_{\mathcal{D}}\mathcal{B}$ is a
dg-category (\cite{benbassatB2012}, Section 4). Its objects are triples
\[
x=(M,N,\phi)\qquad M\in\mathcal{B},\quad N\in\mathcal{C},\quad\phi
\in\Hom_{\mathcal{D}}^{0}(G(M),L(N))
\]
such that $\phi$ is closed and becomes invertible in $\mathrm{Ho}(D)$. The
morphisms of degree $i$ from an object $(M_{1},N_{1},\phi_{1})$ to an object
$(M_{2},N_{2},\phi_{2})$ are the triples%
\[
(\mu,\nu,\gamma)\in\Hom_{\mathcal{B}}^{i}(M_{1},M_{2})\oplus\Hom_{\mathcal{C}%
}^{i}(N_{1},N_{2})\oplus\Hom_{D}^{i-1}(G(M_{1}),L(N_{2})),
\]
and the differential is
\[
d(\mu,\nu,\gamma)=(d\mu,d\nu,d\gamma+\phi_{2}G(\mu)-(-1)^{i}L(\nu)\phi_{1}).
\]

\subsection*{Definition of $\mathcal{DM}(k,\mathbb{Z})$}

The full category $\mathcal{D}\mathcal{M}_{\mathrm{et}}(k,\mathbb{Q})$ of
\'{e}tale dg-motives is defined to be the localization of $\mathcal{DM}%
_{\text{et}}^{\text{eff}}(k,\mathbb{Q})$ by the Tate motive
(\cite{beilinsonV2008}, Section 6.1). Since the Tate motive maps under
$\mathrm{Crys}$ to an invertible object in $\mathcal{D}_{c}^{b}(K_{\sigma
}[F])$, the dg-realization $\mathrm{Crys}$ automatically extends to
\[
\mathrm{Crys}\colon\mathcal{D}\mathcal{M}_{\mathrm{et}}(k,\mathbb{Q}%
)\rightarrow\mathcal{D}_{c}^{b}(K_{\sigma}[F]).
\]

\begin{remark}
(a) As mentioned earlier, the equivalence $\sigma$ in (\ref{eq30}) has a dg
enhancement; this clearly extends to the non-effective categories. So the
definitions of \cite{beilinsonV2008} for $\mathcal{DM}(k,\mathbb{Q}{})$ of
Voevodsky motives (Nisnevich topology) can be replaced with the \'{e}tale version.

(b) (\cite{beilinsonV2008}, Section 6.1). Explicitly, an object of
$\mathcal{DM}_{\mathrm{et}}(k,\mathbb{Q}{})$ is represented as $M(a)$,
$M\in\mathcal{D}{}\mathcal{M}{}_{\mathrm{et}}^{\mathrm{eff}}(k,\mathbb{Q}{})$,
$a\in\mathbb{Z}{}$, and%
\[
\Hom_{\mathcal{D}{}\mathcal{M}{}_{\mathrm{et}}(k,\mathbb{Q}{})}%
(M(a),N(b))=\varinjlim_{n}\Hom_{\mathcal{D}{}\mathcal{M}{}_{\mathrm{et}%
}^{\mathrm{eff}}(k,\mathbb{Q}{})}(M(a+n),N(b+n));
\]
here the inductive limit is taken as $n\rightarrow+\infty$ with $a+n$, $b+n$
nonnegative. Cf. the definition of the category $\mathbf{Crys}(k)$, just
before Lemma 1.7, in \cite{milneR2004}.
\end{remark}

We write $\mathrm{Crys}$ again for the composite%
\[
\mathcal{D}\mathcal{M}_{\mathrm{et}}(k,\mathbb{Z}[p^{-1}])\rightarrow
\mathcal{D}\mathcal{M}_{\mathrm{et}}(k,\mathbb{Q})\overset{\mathrm{Crys}%
}{\longrightarrow}\mathcal{D}(K_{\sigma}[F]).
\]
Let $s\colon\mathcal{D}_{c}^{b}(R)\rightarrow\mathcal{D}_{c}^{b}(K_{\sigma
}[F])$ be the dg functor sending a complex of graded $R$-modules to the
associated simple complex tensored with $K$. We define the \emph{dg category
of motives} over $k$ to be the fibred product%
\[
\mathcal{DM}(k,\mathbb{Z})=\mathcal{DM}_{\mathrm{et}}(k,\mathbb{Z}%
[p^{-1}])\times_{\mathcal{D}_{c}^{b}(K_{\sigma}[F])}\mathcal{D}{}_{c}^{b}(R).
\]
In other words, the following diagram is cartesian:%
\[
\begin{tikzcd}
\mathcal{DM}(k,\mathbb{Z})\arrow{d}\arrow{r}&\mathcal{D}^b_c(R)\arrow{d}{s}\\
\mathcal{DM}_{\text{et}}(k,\mathbb{Z}[p^{-1})\arrow{r}{\mathrm{Crys}}&\mathcal{D}_c^b(K_{\sigma}[F])
\end{tikzcd}
\]
In the diagram, the vertical arrows are contravariant and covariant
respectively, and the horizontal arrows are covariant and contravariant
respectively. We also write $\mathcal{DM}(k)$ for $\mathcal{DM}(k,\mathbb{Z}%
{})$.

By definition, objects of $\mathcal{DM}(k,\mathbb{Z})$ are triples
$(M,N,\phi)$ where $N$ is an object of $\mathcal{DM}_{\mathrm{et}%
}(k,\mathbb{Z}[p^{-1}])$, $M$ is an object of $\mathcal{D}{}(R)$ and $\phi
\in\mathcal{D}_{c}^{b}(K_{\sigma}[F])(s(M),\mathrm{Crys}(N))$ is closed and
becomes invertible in $\mathrm{Ho(}\mathcal{D}(K_{\sigma}[F]))=D(K_{\sigma
}[F])$. The homotopy category of $\mathcal{DM}(k,\mathbb{Z})$ is the
triangulated category $\mathsf{\mathsf{DM}}(k)=\mathsf{\mathsf{DM}%
}(k,\mathbb{Z})$ of integral motivic complexes over $k$.

\begin{aside}
\label{v8}If we take the complex dual to Besser's complex, we will get a
covariant rigid realization (Borel-Moore homology, dual to cohomology) and
then all the arrows in the above diagram will be covariant. But then the
motive defined by a smooth projective variety will be the triple
$(DR\Gamma(X,W\Omega_{X}),hX,-)$ where the first component is the dual of the
complex computing the de Rham-Witt cohomology of $X$, or the triple
$(R\Gamma(X,W\Omega_{X}),DhX,-)$ where we use duality in Voevodsky's category.
\end{aside}

\subsection*{Mayer-Vietoris sequence}

Fix a commutative ring $A$ and an $A$-linear dg-category $\mathcal{B}$. The
dg-category $\mathcal{B}$-$\mathrm{Mod}$ of $\mathcal{B}$-modules is defined
to be the category of dg-functors from $\mathcal{B}$ to the dg-category of
complexes of $A$-modules. Every object $y$ of $\mathcal{B}$ defines a
dg-module $x\rightsquigarrow\Hom_{\mathcal{B}}(x,y)$. There is a natural
embedding $\mathcal{B}\rightarrow\mathcal{B}$\textrm{-}$\mathrm{Mod}$ (Yoneda)
given by $y\mapsto\Hom_{\mathcal{B}}(-,y)$. Every dg-functor $f\colon
\mathcal{B}^{\prime}\rightarrow\mathcal{B}$ induces a pullback dg-functor
$f^{\ast}\colon\mathcal{B}$-$\mathrm{Mod}\rightarrow\mathcal{B}^{\prime}%
$-$\mathrm{Mod}$.

Every $y$ in $\mathcal{B}$ defines the dg-module $\Hom_{\mathcal{B}}(-,y)$
whereas the bi-functor $\Hom_{\mathcal{B}}(-,-)$ defines a canonical bimodule
$M_{\mathcal{B}}$. See Section 2 of \cite{tabuada2013} for more details.

Consider the two bimodules $U$ and $V$ on $\mathcal{DM}(k,\mathbb{Z})$ defined
as follows. Say $x=(M,N,\phi)$ and $x^{\prime}=(M^{\prime},N^{\prime}%
,\phi^{\prime})$ are objects of $\mathcal{DM}(k,\mathbb{Z})$. We define
\[
U(x,x^{\prime})=\Hom_{\mathcal{D}_{c}^{b}(K_{\sigma}[F])}(s(M),\mathrm{Crys}%
(N^{\prime}))[1]
\]
and
\[
V(x,x^{\prime})=\Hom_{\mathcal{D}{}_{c}^{b}(R)}(M,M^{\prime})\oplus
\Hom_{\mathcal{DM}_{\mathrm{et}}(k,\mathbb{Z}[p^{-1}])}(N,N^{\prime}).
\]
The definition of the morphisms in the homotopy fibred dg category provides
the exact sequence of $\mathcal{DM}(k,\mathbb{Z})$-bimodules
\[
0\rightarrow U\rightarrow M_{DZ}\rightarrow V\rightarrow0.
\]

Applied to any pair $x,x^{\prime}$ of objects in $\mathcal{DM}(k,\mathbb{Z})$,
we get a short exact sequence of $\mathbb{Z}$-graded complexes
\[
0\rightarrow U(x,x^{\prime})\rightarrow\Hom_{\mathcal{DM}(k,\mathbb{Z}%
)}(x,x^{\prime})\rightarrow V(x,x^{\prime})\rightarrow0
\]
and the associated (Mayer-Vietoris) long exact sequence
\begin{align*}
\cdots &  \rightarrow\mathrm{Ext}_{\mathcal{DM}(k)}^{i}(x,x^{\prime
})\rightarrow\mathrm{Ext}_{\mathcal{D}{}_{c}^{b}(R)}^{i}(M,M)\oplus
\mathrm{Ext}_{\mathcal{DM}_{\mathrm{et}}(k,\mathbb{Z}[p^{-1}{}])}%
^{i}(N,N^{\prime})\rightarrow\\
&  \rightarrow\mathrm{Ext}_{\mathcal{D}_{c}^{b}(K_{\sigma}[F])}^{i}%
(sM,\mathrm{Crys(}N^{\prime}))\rightarrow\mathrm{Ext}_{\mathcal{DM}%
(k,\mathbb{Z}{})}^{i+1}(x,x^{\prime})\rightarrow\cdots.
\end{align*}
Note: the homotopy fibre-product is designed to give the Mayer-Vietoris sequence!

\subsection{Properties of $\mathsf{DM}(k)$}

We now develop some of the properties of $\mathsf{DM}(k)$.

\subsubsection{Triangulated structure}

All the dg-categories involved in the construction of $\mathcal{DM}%
(k,\mathbb{Z}{})$ are dg-enhancements of triangulated categories, and are
therefore pretriangulated. Recall that a dg-category $\mathcal{C}{}$ is
(strongly) pretriangulated if

\begin{itemize}
\item for each object $A$ of $\mathcal{C}$ and $m\in\mathbb{Z}{}$, there
exists an object, denoted $A[m]$, representing the functor $C\rightsquigarrow
\Hom^{m}(C,A)$), so%
\[
\Hom(C,A[m])=\Hom^{m}(C,A)\text{, all }C\in\ob(\mathsf{C});
\]

\item for each morphism $f\colon A\rightarrow B$ in $\mathcal{C}$ with $d\circ
f=0$, there exists an object, denoted Cone$(f)$, representing the functor
sending each $C\in\ob(\mathsf{C})$ to the cone on%
\[
\Hom(C,A)\overset{f\circ-}{\longrightarrow}\Hom(C,B)\text{.}%
\]

\end{itemize}

\noindent If $\mathsf{C}$ is strongly pre-triangulated, then $\mathrm{Ho}%
(\mathsf{C})$ has a translation functor, namely, $A\rightsquigarrow A[1]$, and
a class of distinguished triangles, namely, those isomorphic to one of the
form%
\[
A\overset{f}{\longrightarrow}B\rightarrow\mathrm{Cone}(f)\rightarrow A[1].
\]
With this structure, $\mathrm{Ho}(\mathsf{C})$ becomes a triangulated category
(\cite{beilinsonV2008}, 1.5.4).

For any object $x=(N,N,\phi)$ of $\mathcal{DM}(k,\mathbb{Z}{})$ and any
integer $n$, the object $x[n]$ representing the shift exists, and is given by
$(M[n],N[n],\phi_{n})$. Similarly, one can check that the cone of any map
$f\colon x\rightarrow x^{\prime}$ is representable. Therefore $\mathsf{DM}(k)$
is a triangulated category

\subsubsection{Motives of smooth varieties}

The category $\mathsf{DM}(k)$ contains an identity object $\1=(W,N(\Spec
k),\id)$ where $N(\Spec k)$, the motive of a point, is the identity object of
$\mathcal{DM}_{\mathrm{et}}^{\mathrm{eff}}(k,\mathbb{Z}{}[p^{-1}])$
(\cite{beilinsonV2008}, \S 2.2, \S 2.3). The object $W$ is the identity object
of $\mathsf{D}_{c}^{b}(R)$, and $s(W)=K$ considered as a complex in degree
zero with $F=\sigma$. This is \textrm{Crys}$(N(\Spec(k))$.

For any smooth proper variety $X$, consider the object%
\[
(R\Gamma(X,W\Omega),N(X),\phi)\in\mathcal{DM}(k,\mathbb{Z}{})
\]
where the first term is a suitable complex computing the de Rham-Witt
cohomology of $X$, the second term is the motive of $X$, and the third term is
the canonical isomorphism between de Rham-Witt cohomology (tensored with $K$)
and rigid cohomology (Berthelot 1986). The resulting object $h_{X}$ in
$\mathcal{DM}(k,\mathbb{Z}{})$ is well defined because any two suitable
complexes computing the de Rham-Witt cohomology are quasi-isomorphic.

For any smooth variety $X$, similarly consider the object $(C(X),N(X),\phi)$
where $N(X)$ is the motive of $X$ (see above) and $C(X)$ is the object of
$\mathsf{D}_{c}^{b}(R)$ attached to $X$ in \S 6 of \cite{milneR2013}. The map
$\phi$ comes from the canonical isomorphism between logarithmic de Rham-Witt
cohomology and rigid cohomology (\cite{nakk2008}, \cite{nakk2012}).

\subsubsection{Tensor structure and internal Hom's}

We expect that the homotopy categories $\mathsf{DM}(k)$ and $\mathsf{DM}%
^{\mathrm{eff}}(k)$ are tensor triangulated categories.

We sketch the proof for $\mathsf{DM}^{\mathrm{eff}}(k)$. There is a homotopy
tensor structure on $\mathcal{DM}^{\mathrm{eff}}(k,\mathbb{Z}{}[p^{-1}])$
(\cite{beilinsonV2008}, Section 2.2, 2.3; \cite{volog2012}, line before Lemma
2.3). There are clearly also homotopy tensor structures on $\mathcal{D}{}%
_{c}^{b}(R)$ and $\mathcal{D}{}_{c}^{b}(K_{\sigma}[F])$. The functors $s$ and
$\mathrm{Crys}$ are compatible with the tensor structure (up to homotopy) ---
for \textrm{Crys }this is a consequence of the K\"{u}nneth property for rigid
cohomology. This suffices to endow $\mathsf{DM}^{\mathrm{eff}}(k)$ with the
structure of a tensor triangulated category.

More precisely, the cartesian product of schemes induces a tensor product
structure on $\mathcal{DM}_{\mathrm{et}}^{\mathrm{eff}}(k,\mathbb{Z}{}%
[p^{-1}])$ (\cite{volog2012}, line before Lemma 2.3). The category is
generated by the motives $N(X)$ of smooth varieties. Thus, the homotopy tensor
structure on $\mathcal{DM}^{\mathrm{eff}}(k,\mathbb{Z}{})$, which defines the
tensor product structure on $\mathsf{DM}^{\mathrm{eff}}(k)$ is determined by
$h(X)\otimes h(Y)\simeq h(Y)\otimes h(X)$.

Similar comments apply to the internal Hom. The internal Hom of $N(X)$ and
$N(Y)$ in $\mathcal{DM}^{\mathrm{eff}}(k,\mathbb{Z}{}[p^{-1}])$ is defined by
the following equality for all smooth varieties $Z,$%
\[
\Hom(N(Z),R\underline{\Hom}(N(X),N(Y))=\Hom(N(Z\times X),N(Y))
\]
(\cite{beilinsonV2008}, Section 2.2).

On the other hand, the existence of a homotopy tensor structure on
$\mathcal{DM}(k,\mathbb{Z}{}[p^{-1}])$ does not seem to be known (ibid.,
remark at the end of Section 6.1). Therefore, it does not seem to be known
that $\mathsf{DM}(k)$ has a natural tensor structure.

\subsubsection{Tate twist}

The Tate object $R(1)$ in $\mathcal{DM}_{\mathrm{et}}^{\mathrm{eff}%
}(k,\mathbb{Z}{}[p^{-1}])$ is determined by the property%
\[
N(\Spec k)\oplus R(1)[1]=N(\mathbb{G}_{m})
\]
where $N(X)$ is the motive of the smooth variety $X$ (\cite{beilinsonV2008},
Section 2.2). The Tate twist is given by $F(1)=F\otimes R(1)$. Recall that the
rigid cohomology groups of $\mathbb{G}_{m}$ are%
\[
H_{\mathrm{rig}}^{0}(\mathbb{G}_{m})=K(0),\quad H_{\mathrm{rig}}%
^{1}(\mathbb{G}_{m})=K(1)
\]
where $K(m)$ is the $F$-isocrystal $K$ with $F=p^{m}\sigma$. Thus
$\mathrm{Crys}(R(1))=K(1)$.

The Tate object $E(1)$ in $\mathsf{DM}(k)$ is the triple%
\[
E(1)=(W(1),R(1),\phi)
\]
where $W(1)$ is the Tate twist of the identity object of $\mathsf{D}_{c}%
^{b}(R)$ (\cite{milneR2005}, Section 2, or \cite{milneR2013}, 1.6) and $\phi$
is the natural isomorphism between $sW(1)=K(1)$ and \textrm{Crys}$(R(1))$. The
Tate twist on $\mathsf{DM}(k)$ is given by $M(1)=M\otimes E(1)$ for $M$ an
object of $\mathsf{DM}(k)$.

\subsubsection{Realization functors}

There are realization functors $r_{l}$ on $\mathsf{DM}(k)$ for all primes $l$.

The dg functor%
\[
pr_{2}\colon\mathcal{DM}(k,\mathbb{Z}{})\rightarrow\mathcal{DM}(k,\mathbb{Z}%
{}[p^{-1}]),\quad(M,N,\phi)\rightsquigarrow N,
\]
induces a functor $pr_{2}\colon\mathsf{DM}(k)\rightarrow\mathsf{DM}%
_{\mathrm{et}}(k,\mathbb{Z}{}[p^{-1}])$ which is clearly an exact functor
between these triangulated categories. For each prime $l\neq p$, there is an
\'{e}tale realization functor
\[
r_{l}^{\prime}\colon\mathsf{DM}_{\mathrm{et}}(k,\mathbb{Z}{}[p^{-1}%
])\rightarrow\mathsf{D}_{c}^{b}(k,\mathbb{Z}{}_{l})
\]
which is an exact functor of tensor triangulated categories (\cite{ayoub2013}%
). We define the $l$-adic realization $r_{l}$ to be the composite%
\[
r_{l}^{\prime}\circ pr_{2}\colon\mathsf{DM}(k)\rightarrow\mathsf{D}_{c}%
^{b}(k,\mathbb{Z}{}_{l}).
\]
As $r_{l}^{\prime}$ is an exact functor of triangulated categories, so also is
$r_{l}$.

The dg functor%
\[
pr_{1}\colon\mathcal{DM}(k,\mathbb{Z}{})\rightarrow\mathcal{D}_{c}%
^{b}(R),\quad(M,N,\phi)\rightsquigarrow M,
\]
induces a functor $\mathsf{DM}(k)\rightarrow\mathsf{D}_{c}^{b}(R)$ on the
associated triangulated categories. This is the realization functor $r_{p}$.

\subsubsection{Rigidity}

\begin{proposition}
\label{v9}For all $l$, including $l=p$, the realization functor $r_{l}$
defines an equivalence on the subcategories of objects killed by $l^{n}$.
\end{proposition}

\begin{proof}
This follows from the Mayer-Vietoris sequence and \cite{voevodsky2000}, 3.3.3.
\end{proof}

We say that an object $X$ of a triangulated category $\mathsf{C}$ is
$n$-torsion ($n\in\mathbb{Z}$) if the abelian group $\Hom(X,Y)$ is killed by
$n$ for all objects $Y$ of $\mathsf{C}$. We say that $X$ is $l$-power torsion
for a prime $l$ if every element of $\Hom(X,Y)$ is killed by a power of $l$,
all $Y$ in $\mathsf{C}$.

\begin{proposition}
\label{v10}For any $x=(M,N,\phi)$ and $x^{\prime}=(M^{\prime},N^{\prime}%
,\phi^{\prime})$ in $\mathcal{DM}(k,\mathbb{Z}{})$, and every positive integer
$n$, the map%
\[
\Hom_{\mathsf{DM}(k)}(x,x^{\prime})\otimes\mathbb{Z}{}/p^{n}\mathbb{Z}%
{}\rightarrow\Hom_{\mathsf{D}_{c}^{b}(R)}(M,M^{\prime})\otimes\mathbb{Z}%
{}/p^{n}\mathbb{Z}{}%
\]
induced by $r_{p}$ is an isomorphism.
\end{proposition}

\begin{proof}
Omitted.
\end{proof}

\subsection*{Motivic complexes over $\mathbb{F}{}_{q}$}

We need a category $\mathsf{\mathsf{DM}}(\mathbb{F}{}_{q})$ such that the
functor
\[
R\Hom(\1,-)\colon\mathsf{\mathsf{DM}}(\mathbb{F}{}_{q})\rightarrow
\mathsf{D}(\mathbb{Z}{}\mathbb{)}%
\]
\textsf{ factors through }$\mathsf{D}(\mathbb{Z}\Gamma_{0})$. We can either
proceed as above, but with the Weil-\'{e}tale topology for the \'{e}tale
topology or, more simply, as in \S 6.

\subsection{Applications to algebraic varieties}

To give an object of $\mathcal{DM}(k)$ amounts to giving objects of
$\mathcal{DM}_{\mathrm{et}}(k,\mathbb{Z}{}[p^{-1}])$ and $\mathcal{D}_{c}%
^{b}(R)$ together with an isomorphism between their realizations in
$\mathcal{D}{}_{c}^{b}(K_{\sigma}[F])$. In the final section of
\cite{milneR2013}, we explained how to attach an object of $\mathsf{D}_{c}%
^{b}(R)$ to an arbitrary variety, a variety with log structure, a
Deligne-Mumford stack, etc.. All of the statements there carry over
\textit{mutatis mutandis} to the present situation.

\begin{plain}
\label{v11}Choudhury (2012)\nocite{choudhury2012} attaches an object in
Voevodsky's category $\mathsf{\mathsf{DM}}^{\mathrm{eff}}(k,\mathbb{Q}{})$ to
a smooth Deligne-Mumford stack.
\end{plain}

\begin{plain}
\label{v12}\cite{voevodsky2010} attaches an object in his category
$\mathsf{\mathsf{DM}}_{-}^{\mathrm{eff}}(k)$ to a smooth simplicial scheme
over a field $k$.
\end{plain}

\section{Motivic complexes for rational Tate classes}

In this section, we sketch the construction of a category $\mathsf{DM}(k)$ for
which the \textquotedblleft Tate conjecture\textquotedblright\ is
automatically true. The construction requires only the rationality conjecture
of \cite{milne2009}, which is much weaker than the Tate conjecture.

\subsection{Rational Tate classes}

In this subsection, we review part of \cite{milne2009}. Let $\mathbb{Q}%
{}^{\mathrm{al}}$ denote the algebraic closure of $\mathbb{Q}{}$ in
$\mathbb{C}{}$. Fix a $p$-adic prime $w$ of $\mathbb{Q}{}^{\mathrm{al}}$, and
let $\mathbb{F}{}$ be its residue field. Then $\mathbb{F}{}$ is an algebraic
closure of $\mathbb{F}{}_{p}$. We assume that the reader is familiar with the
theory of absolute Hodge classes (\cite{deligne1982}).

\begin{rconjecture}
\label{c0}Let $A$ be an abelian variety over $\mathbb{Q}{}^{\mathrm{al}}$ with
good reduction to an abelian variety $A_{0}$ over $\mathbb{F}{}$, and let
$d=\dim(A)$. An absolute Hodge class $\gamma$ of codimension $i$ on $A$
defines (by specialization) classes $\gamma_{l}\in H^{2i}(A_{0},\mathbb{Q}%
_{l}(i))$ for all $l\neq p$ and $\gamma_{p}\in H_{\mathrm{crys}}^{2i}%
(A_{0}/W)(i)_{\mathbb{Q}{}}$. Let $D_{1},\ldots,D_{d-i}$ be divisors on
$A_{0}$, and let $\delta_{1}(l),\ldots,\delta_{d-i}(l)$ denote their
$l$-cohomology classes. The conjecture says that%
\[
\gamma_{l}\cdot\delta_{1}(l)\cdot\cdots\cdot\delta_{d-i}(l),
\]
which a priori lies in $\mathbb{Q}{}_{l}$ or $\left(  \mathbb{Q}%
{}^{\mathrm{al}}\right)  _{w}$, is a rational number independent of $l$.
\end{rconjecture}

Now let $\mathcal{S}{}$ be a class of smooth projective algebraic varieties
over $\mathbb{F}{}$ that is closed under passage to a connected component and
under the formation of finite products and disjoint unions. We assume that
$\mathcal{S}{}$ contains the class $\mathcal{S}{}_{0}$ of all abelian
varieties over $\mathbb{F}{}$, and that the Frobenius elements of the
varieties in $\mathcal{S}{}$ act semisimply on cohomology. For an $X$ in
$\mathcal{S}{}$, we let $H_{\mathbb{A}{}}^{2i}(X)(i)$ denote the restricted
product of the cohomology groups $H^{2i}(X,\mathbb{Q}{}_{l}(i))$ for $l\neq p$
with $H_{\mathrm{crys}}^{2i}(X/W)_{\mathbb{Q}{}}$.

\begin{definition}
\label{c1}A family $(\mathcal{R}{}^{\ast}(X))_{X\in\mathcal{S}{}}$ with each
$\mathcal{R}{}^{\ast}(X)$ a graded $\mathbb{Q}{}$-subalgebra of $H_{\mathbb{A}%
{}}^{2\ast}(X)(\ast)$ is a \emph{good theory of rational Tate classes on
}$\mathcal{S}{}$ if it satisfies the following conditions:

(R1) for every regular map $f\colon X\rightarrow Y$ of varieties in
$\mathcal{S}$, $f^{\ast}$ maps $\mathcal{R}{}^{\ast}(Y)$ into $\mathcal{R}%
{}^{\ast}(X)$ and $f_{\ast}$ maps $\mathcal{R}{}^{\ast}(X)$ into
$\mathcal{R}{}^{\ast}(Y)$;

(R2) for every $X$ in $\mathcal{S}{}$, $\mathcal{R}{}^{1}(X)$ contains the
divisor classes;

(R3) for all CM abelian varieties $A$ over $\mathbb{Q}{}^{\mathrm{al}}$, the
absolute Hodge classes on $A$ map to elements of $\mathcal{R}{}^{\ast}(A_{0})$
under the specialization map;

(R4) For all varieties $X$ in $\mathcal{S}{}$, the $\mathbb{Q}{}$-algebra
$\mathcal{R}{}^{\ast}(X)$ is of finite degree, and the $l$-primary components
of every element of $\mathcal{R}{}^{\ast}(X)$ are $l$-adic Tate classes.
\end{definition}

Recall that an abelian variety $A$ is CM if its endomorphism algebra
$\End(A)_{\mathbb{Q}{}}$ contains an \'{e}tale subalgebra of degree $2\dim A$
over $\mathbb{Q}{}$, and that such an abelian variety over $\mathbb{Q}%
{}^{\mathrm{al}}$ has good reduction at all the primes of $\mathbb{Q}%
{}^{\mathrm{al}}$. For the space $\mathcal{T}{}_{l}^{i}(X)$ of Tate classes in
$H^{2i}(X,\mathbb{Q}{}_{\ell}(i))$ or $H_{\mathrm{crys}}^{2i}%
(X/W)(i)_{\mathbb{Q}}$, see ibid. pp.112--113.

\begin{theorem}
\label{c2}(a) There exists at most one good theory of rational Tate classes on
$\mathcal{S}{}$.

(b) There exists a good theory of rational Tate classes on $\mathcal{S}_{_{0}%
}$ if the rationality conjecture holds.

(c) The Tate conjecture holds for every good theory of rational Tate classes,
i.e., the maps $\mathcal{R}{}^{i}(X)\otimes\mathbb{Q}{}_{l}\rightarrow
\mathcal{T}{}_{l}^{i}(X)$ induced by the projection maps are isomorphisms for
all $X$, $i$, and $l$.
\end{theorem}

\begin{proof}
Ibid. Theorem 3.3; Theorem 4.5; Theorem 3.2.
\end{proof}

\subsection{The category $\mathsf{DM}(k)$}

We assume that there exists a good theory of rational Tate classes for some
class $\mathcal{S}{}$ as in the last subsection. Grothendieck's construction
now gives us a Tannakian $\mathbb{Q}{}$-linear category $\Mot(k,\mathbb{Q}{})$
of motives, and we define $\mathsf{DM}(k,\mathbb{Q}{})$ to be its derived category.

Let $\mathbb{Z}{}^{p}=\prod\nolimits_{l\neq p}\mathbb{Z}{}_{l}$. It is
generally believed that the constructions of the $\ell$-adic triangulated
category $\mathsf{D}_{c}^{b}(X,\mathbb{Z}{}_{\ell})$ (\cite{deligne1980},
\cite{ekedahl1990}, \cite{bhattS2013}) will generalize to give an adic
category $\mathsf{D}_{c}^{b}(X,\mathbb{Z}{}^{p})$, but as far as we know, no
proof has been written out. Instead, we give an ad hoc construction of
$\mathsf{D}_{c}^{b}(k,\mathbb{Z}{}^{p})$. Let $\Gamma$ be the absolute Galois
group of $k$. Let $\Lambda=\mathbb{Z}{}^{p}$ and let $\Lambda_{m}=\mathbb{Z}%
{}/m\mathbb{Z}{}$ when $(m,p)=1$. The inverse systems $M=(M_{m})_{m}$ in which
$M_{m}$ is a continuous $\Lambda_{m}\Gamma$-module form an abelian category
whose derived category we denote by $\mathsf{D}(k,\Lambda_{\bullet})$. As in
(\ref{a21a}), there is an obvious \textquotedblleft
completion\textquotedblright\ functor $M\rightsquigarrow\hat{M}\colon
\mathsf{D}(k,\Lambda_{\bullet})\rightarrow\mathsf{D}(k,\Lambda_{\bullet})$. We
define $\mathsf{D}_{c}^{b}(k,\mathbb{Z}{}^{p})$ to be the full subcategory of
$\mathsf{D}(k,\Lambda_{\bullet})$ consisting of complexes $M$ such that
$M\simeq\hat{M}$ and $(\mathbb{Z}{}/\ell\mathbb{Z}{})\otimes_{\Lambda
_{\bullet}}^{L}M$ is a bounded complex with bounded finite-dimensional
cohomology for all primes $\ell\neq p$. The functor $(M_{n})_{n}%
\rightsquigarrow$ $\varprojlim M_{n}(k^{\mathrm{sep}})$ defines a functor%
\[
\alpha\colon\mathsf{D}_{c}^{b}(k,\mathbb{Z}{}^{p})\rightarrow\mathsf{D}%
(\mathbb{\mathbb{Z}{}}^{p}\Gamma)\text{.}%
\]
It follows from the next lemma that the cohomology groups of $\alpha M$ are
finitely presented $\mathbb{Z}{}^{p}$-modules.

\begin{lemma}
\label{c3}The following conditions on a $\mathbb{Z}^{p}$- module are equivalent.

\begin{enumerate}
\item $M$ is of finite presentation.

\item The $\mathbb{Z}{}_{\ell}$-module $\mathbb{Z}{}_{\ell}\otimes
_{\mathbb{Z}{}^{p}}M$ is finitely generated for all $\ell$, the natural map
$M\rightarrow\prod\nolimits_{\ell\neq p}\left(  \mathbb{Z}{}_{\ell}%
\otimes_{\mathbb{Z}{}^{p}}M\right)  $ is an isomorphism, and $\dim
_{\mathbb{F}{}_{\ell}}(M/\ell M)$ is bounded (independently of $\ell$).

\item The natural map $M\rightarrow\varprojlim_{(m,p)}M/mM$ is an isomorphism,
and $\dim_{\mathbb{F}{}_{\ell}}(M/\ell M)$ is bounded (independently of $\ell$).
\end{enumerate}
\end{lemma}

\begin{proof}
Elementary exercise.
\end{proof}

There is an exact functor of triangulated categories $\mathsf{DM}%
(k,\mathbb{Q}{})\rightarrow\mathsf{D}_{c}^{b}(\mathbb{Z}^{p})_{\mathbb{Q}{}%
}\times\mathsf{D}_{c}^{b}(R)_{\mathbb{Q}{}}$, and we define $\mathsf{DM}(k)$
to be the universal object fitting into a diagram%
\[
\begin{tikzcd}
\mathsf{DM}(k)\arrow{r}\arrow{d}&\mathsf{DM}(k,\mathbb{Q})\arrow{d}\\
\mathsf{D}_{c}^{b}(\mathbb{Z}^{p})\times\mathsf{D}_{c}^{b}(R)\arrow{r}
&\mathsf{D}_{c}^{b}(\mathbb{Z}^{p})_{\mathbb{Q}}\times\mathsf{D}_{c}^{b}(R)_{\mathbb{Q}}.
\end{tikzcd}
\]
More precisely, we form the homotopy fibred product of the natural dg
enhancements of the categories, and then pass to the associated triangulated
category (see \S 7). Conjecturally $\mathsf{DM}(k)$ is independent of the
choice of $\mathcal{S}{}$ containing $\mathcal{S}{}_{0}$.

For this category, the \textquotedblleft Tate conjecture\textquotedblright%
\ holds automatically, and so we can apply the results of \S 6.

\subsubsection{Final note}

Contrary to our earlier claim, (\ref{a1}) is not in fact the ultimate
conjecture. The zeta function of an Artin stack is also defined, but it is a
power series in $t$ (not a rational function). To accomodate Artin stacks, one
will need to state a conjecture for a category of \textit{unbounded} motivic
complexes, but we (the authors, and perhaps also the reader) are already exhausted.

\subsubsection{Acknowledgements}

The second author wishes to thank T. Geisser, G. Tabuada, and V. Vologodsky
for their help with Section 7.

\bibliographystyle{cbe}
\bibliography{D:/Current/refs}

\bigskip James S. Milne,

Mathematics Department, University of Michigan, Ann Arbor, MI 48109, USA,

Email: jmilne@umich.edu

Webpage: \url{www.jmilne.org/math/}

\bigskip Niranjan Ramachandran,

Mathematics Department, University of Maryland, College Park, MD 20742, USA,

Email: atma@math.umd.edu,

Webpage: \url{www.math.umd.edu/~atma/}
\end{document}